\documentclass[11pt]{amsart}
\usepackage{a4}
\usepackage{fancybox}
\usepackage{epsfig}
\usepackage[english]{babel}
\usepackage{textcomp}
\usepackage{amssymb}
\usepackage{amsmath}
\usepackage{graphicx}

\newcommand{\1}{\mathbb{1}}

\newcommand{\N}{\mathbb{N}}

\newcommand{\C}{\mathbb{C}}
\newcommand{\D}{\mathbb{D}}

\newcommand{\K}{\mathcal{K}}

\newcommand{\ds}{\displaystyle}

\newtheorem{thank}{\ \ \ Acknowledgment}
\newcounter{tictac}

\def\1{\,\rlap{\mbox{\small\rm 1}}\kern.15em 1}

\def\build#1_#2^#3{\mathrel{\mathop{\kern 0pt#1}\limits_{#2}^{#3}}}
\def\tend#1#2{\build\hbox to 12mm{\rightarrowfill}_{#1\rightarrow #2}^{ }}

\def\converge#1#2#3#4{\build\hbox to
#1mm{\rightarrowfill}_{#2\rightarrow #3}^{\hbox{\scriptsize #4}}}

\newtheorem{thm}{Theorem}[section]
\newtheorem{prop}[thm]{Proposition}
\newtheorem{lemm}[thm]{Lemma}

\newtheorem{conj}[thm]{Conjecture}

\newtheorem{rem}[thm]{Remark}

\newcommand{\beq}{\begin{equation}}
\newcommand{\eeq}{\end{equation}}

\begin{document}
\title[On the spectrum of stochastic perturbations of the shift]
{On the spectrum of stochastic perturbations of the shift and Julia sets}
\author{E. H. El Abdalaoui}
\address { Department of Mathematics, University
of Rouen, LMRS, UMR 60 85, Avenue de l'Universit\'e, BP.12, 76801
Saint Etienne du Rouvray - France}
\email{elhoucein.elabdalaoui@univ-rouen.fr }
\author {A. Messaoudi }
\address{Departamento de Matem\'atica, IBILCE-UNESP, Rua Cristov\~o Colombo,
2265, CEP 15054-0000, S\~ao
Jos\'e de Rio Preto-SP, Brasil}
\email{messaoud@ibilce.unesp.br}
\footnote{Research partially
supported by French-Brasilian Cooperation (French CNRS and Brasilian CNPq)
and Capes-Cofecub Project 661/10 .
The second author was Supported by Brasilian CNPq grant 305043/2006-4}
\maketitle
{\renewcommand\abstractname{Abstract}

\date{12 september  2011}

\begin{abstract}
We extend the Killeen-Taylor study in \cite{KT} by investigating in different Banach spaces 
($\ell^\alpha(\N), c_0(\N),c_c(\N)$) 
the point, continuous and residual spectra of stochastic perturbations of the shift operator associated to 
the stochastic adding machine in base $2$ and in Fibonacci base. For the base $2$, the spectra are connected to the  Julia set of a quadratic map.
In the Fibonacci case, the spectra involve  the  Julia set of  an endomorphism of $\C^2$.

\vspace{8cm}

\hspace{-0.7cm}{\em AMS Subject Classifications} (2000): 37A30, 37F50, 47A10, 47A35.\\

{\em Key words and phrases:} Markov operator, Markov process, transition operator, 
stochastic perturbations of the shift, stochastic adding machine, Julia sets,   residual spectrum, continuous spectrum.\\
\end{abstract}
\thispagestyle{empty}
\newpage
\section{\bf Introduction}
In this paper, we study  in detail the spectrum of some stochastic perturbations of the shift operator 
introduced by Killeen and Taylor in \cite{KT}.
We focus our study on large Banach spaces for which we complete the Killeen-Taylor study. We investigate also
 the case of Fibonacci base, but in this case, we are not able to compute the residual and continuous spectra 
exactly.\\
 We recall that   
in \cite {KT},
 Killeen and Taylor defined the stochastic adding machine as a stochastic perturbation of the shift
in the following way: let $N$ be a nonnegative integer number written in base $2$ as
$N= \sum_{i=0}^{k(N)} \varepsilon_{i}(N)2^{i}$ where
$\varepsilon_{i}(N)=0$ or $1$ for all $i.$ It is known that there
exists an algorithm  that computes the digits of $N+1$. This
algorithm can be described by introducing an auxiliary
 binary "carry" variable $c_i(N)$ for each digit $\varepsilon_{i}(N)$
 by the following manner:

Put $c_{-1}(N+1)=1$ and
$$\varepsilon_{i}(N+1)= \varepsilon_{i}(N)+ c_{i-1} (N+1) {\rm {\quad mod \quad}} (2)$$
$$c_i (N+1) = \left[{\frac{\varepsilon_{i}(N)+ c_{i-1}(N+1)}{2}}\right]$$
where $i \geq 0$ and $[z]$ denote the integer part of $z \in
\mathbb{R}_{+}.$

Let $\{e_{i}(n): i \geq 0, n \in \mathbb{N}\}$ be an independent,
identically distributed family of random variables  which take the value $0$ with probability
$1-p$ and the value $1$ with probability $p$. Let $N$ be an integer.
Given a sequence  $(r_{i}(N))_{i \geq 0} $  of $0$ and $1$ such that
$r_{i}(N)=1$ for finitely many indices $i$, we consider the
sequences  $(r_{i}(N+1))_{i \geq 0}$ and $(c'_{i}(N+1))_{i \geq -1}$
 defined by
  $c'_{-1}(N+1)=1$ and for all  $i \geq 0$
 $$r_{i}(N+1)= r_{i}(N)+ e_{i}(N)c'_{i-1}(N+1)
{\rm {\quad mod \quad}} (2)$$
 $$ c'_i (N+1) = \left[\frac{r_{i}(N)+
e_{i}(N)c'_{i-1}(N+1)}{2}\right],$$

With this we have that a number $\sum_{i=0}^{+\infty}
 r_{i}(N)2^{i}$ transitions to a number \linebreak$\sum_{i=0}^{+\infty}
 r_{i}(N+1)2^{i}$.
 In particular, an integer  $N$ having
a binary representation of the form $\varepsilon_{n}\ldots
\varepsilon_{k+1}0 \underbrace{11\ldots 11}_{k}$ transitions to
$\varepsilon_{n}\ldots \varepsilon_{k+1}1 \underbrace{00\ldots
00}_{k}$ with probability $p^{k+1}$ and a number having binary
representation of the form $\varepsilon_{n}\ldots \varepsilon_{k}
\underbrace{11\ldots 11}_{k}$ transitions to $\varepsilon_{n}\ldots
\varepsilon_{k} \underbrace{00\ldots 00}_{k}$ with probability
$p^{k}(1-p).$
Equivalently, we obtain a Markov  process $\psi(N)$ with state space $\mathbb{N}$ by $\psi(N)= \sum_{i=0}^{+\infty}r_i(N)2^{i}$. The 
corresponding transition operator is denoted by $S_p$ and given in Figure 2.
 
For $p=1$ the transition operator equals the shift operator (cf. Figure 2), hence the stochastic adding machine can be seen as  a stochastic perturbation of the shift operator. It is  also a
  model of Weber law in the context of counter and pacemarker errors. This law is
used in biology and psychophysiology \cite{KT2}.

In [KT], P.R. Killeen and J. Taylor studied the  spectrum of  the
transition operator $S_p$ (of $\psi(N)$)  on  $\ell^{\infty}$. They
proved that the spectrum $\sigma(S_p)$ is equal to the filled Julia
set of the quadratic map $f: \mathbb{C} \mapsto \mathbb{C} $ defined
by: $f(z)= (z- (1-p))^2 / p^2$,  i.e:
$\sigma (S_p)= \{z \in \mathbb{C},\; (f^{n}(z))_{n \geq 0} \mbox { is bounded } \}$
where $f^n$ is the $n$-th iteration of $f$.

In \cite{Messaoudi-Smania}, Messaoudi and Smania defined the stochastic adding machine
 in the Fibonacci base. The corresponding transition operator is given in Figure 6.
Their procedure can be extended to a large class of  adding machine and  is given by  the following manner. Consider  the
Fibonacci sequence $(F_n)_{n \geq 0}$ given by the relation
$$F_0= 1, F_1=2,\;
 F_n= F_{n-1}+ F_{n-2} \;\; \forall n \geq 2.$$
Using the greedy algorithm, we can write every
nonnegative integer $N$ in a unique way as $\displaystyle N= \sum_{i=0}^{k(N)}
\varepsilon_{i}(N)F_i$ where $\varepsilon _{i}(N)= 0$ or $1$ and
$\varepsilon _{i}(N)\varepsilon _{i+1}(N) \ne 11, \;$ for all $i \in
\left\{0,\cdots,k(N)-1\right\}$ (see \cite{Z}).
It is known that the addition of $1$ in the Fibonacci base (adding machine) is
recognized by a finite state automaton (transductor).
In \cite{Messaoudi-Smania}, the authors defined the stochastic adding machine by introducing
a $``$ probabilistic transductor $\textquotedblright$.
 They also computed the point spectrum  of the
transition operator acting in $\ell^{\infty}$ associated to the stochastic adding machine with
respect to the base $(F_n)_{n \geq 0}$. In particular, they showed that  the point
spectrum $\sigma_{pt}(S_p)$ in $\ell^{\infty}$
 is connected to the filled Julia set $J (g)$ of the function $g:
\mathbb{C}^2 \mapsto \mathbb{C}^2 $ defined by:
$$g(x,y)= (\frac{1}{p^{2}}(x- 1+p)(y- 1+p), x).$$
Precisely, they proved that
$$\sigma_{pt}(S_p)=
 \K_p=  \{ \lambda \in
\mathbb {C} \;  (q_n (\lambda))_{n \geq 1} \mbox { is bounded }\},$$
 where $q_{F_{0}}(z)= z,\; q_{F_{1}} (z)=z^2,\; q_{F_{k}} (z)= \ds \frac{1}{p}q_{F_{k-1}} (z)q_{F_{k-2}} (z)-\ds\frac{1-p}{p},$ for all $k \geq 2$
and for all nonnegative integers $n$, we have $q_n (z)= q_{F_{k_{1}}} \ldots q_{F_{k_{m}}} $ where $ F_{k_{1}}+ \cdots+ F_{k_{m}}$ is the Fibonacci representation of $n$.

In particular, $\sigma_{pt}{(S_{p})}$ is contained in  the set
\begin{eqnarray*}
 \mathcal {E}_p &= &
\{
\lambda \in \mathbb{C} \; \vert \; (q_{F_{n}} (\lambda))_{n \geq 1} \mbox { is bounded } \}\\
                & = &
\{
\lambda \in \mathbb{C} \; \vert \; ( \lambda_1, \lambda) \in
J (g)\}
\end{eqnarray*}
where $\lambda_1= 1-p
+\frac{ (1- \lambda- p)^2}{p}.$\\

\begin{center}
\includegraphics[scale=0.6]{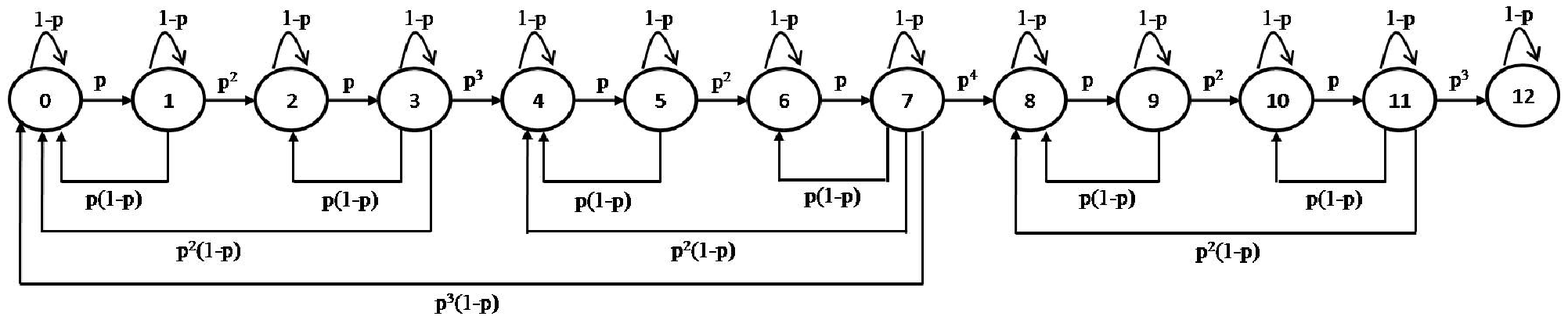}\\
{\footnotesize {Fig.1. Transition graph of stochastic adding machine in base 2}}
\end{center}

Here we investigate the spectrum of the stochastic adding machines in base $2$ and in the Fibonacci 
base in different Banach spaces. In particular,  
we compute exactly the point, continuous and residual spectra of the stochastic adding machine
in base $2$ for the Banach spaces $c_0$, $c$, $\ell^{\alpha},\; \alpha \geq 1$.

For the Fibonacci base, we improve the result in \cite{Messaoudi-Smania} by 
proving that the spectrum of $S_p$ acting on $\ell^{\infty}$ contain  ${\mathcal{E}}_p$. The same result will be proven  
for the Banach spaces  $c_0$, $c$ and $\ell^{\alpha}, \alpha \geq 1$.

The  paper is organized as follows. In section 2, we give some basic facts on  spectral theory. 
In section 3, we state our main results (Theorems 1, 2 and 3). Section 4 contains  
the proof in the case of the base 2 and finally, in  
section 5,  we present the proof in the case of the Fibonacci base.

\section{\bf Basic facts from the spectral theory of operators (see
for instance \cite{Halmos}, \cite{Rudin},\cite{Schechter}, \cite{yoshida})}
Let $E$ be a complex Banach space and $T$ a bounded operator on it.
The spectrum of $T$, denoted by $\sigma(T)$, is the subset of complex numbers
$\lambda $ for which $T-\lambda Id_E$ is not an isomorphism ($Id_E$ is the
identity maps).

As usual we point out that if $\lambda$ is in $\sigma(T)$ then one of the
following assertions hold:
\begin{enumerate}
  \item $T-\lambda Id_E$ is not injective. In this case we say that $\lambda$ is in the point spectrum denoted by
$\sigma_{pt}(T)$.
  \item $T-\lambda Id_E$ is injective, not onto and has dense range. We say
that $\lambda$ is in the continuous spectrum denoted by $\sigma_c(T)$.
  \item $T-\lambda Id_E$ is injective and does not have dense range. We
say that $\lambda$ is in residual spectrum of T denoted by $\sigma_r(T)$.
\end{enumerate}
It follows that $\sigma(T)$ is the disjoint union
\[
    \sigma(T) = \sigma_{pt} (T) \cup \sigma_c (T) \cup \sigma_r (T).
\]
It is well known and it is an easy consequence of Liouville Theorem that
the spectrum of any bounded operator is a non empty compact set of $\C$. There is a connection
 between the spectrum of $T$ and  the spectrum of
the dual operator $T'$ acting on the dual space $E'$  by
$T'~~:~~\phi \mapsto \phi \circ T.$ In particular, we have

\begin{prop}[Phillips Theorem]{}\label{Phillips} Let $E$ be a Banach space and $T$
a bounded operator on it, then $
\sigma(T)=\sigma(T').$
\end{prop}

\noindent We also have  a classical  relation between the point and residual spectra of $T$ and the point spectrum of $T'$.

\begin{prop}\label{residual}
For a bounded operator $T$ we have
$$\sigma_r(T)\subset \sigma_{pt}(T')\subset \sigma_r(T)\cup \sigma_{pt}(T).$$
In particular,
if $\sigma_{pt}(T)$ is an empty set
then $$\sigma_r(T)=\sigma_{pt}(T').$$
\end{prop}

\section{\bf Main results.}

 Our main results are stated in the following three theorems.

{\bf Theorem 1}.
The spectrum of the operator $S_p$ acting on $c_0,\; c$ and $\ell^{\alpha}, \; \alpha  \geq 1$
is equal to the filled Julia set $J(f)$ of the quadratic map $f(z)= (z- (1-p))^2 / p^2$.
 Precisely,
 in $c_0$ (resp. $\ell^{\alpha}, \; \alpha  > 1$), the continuous  spectrum of  $S_p$
is equal to   $J(f)$ and the point and residual spectra are empty.
In $c$, the point spectrum is equal $\{1\}$, the residual spectrum is  empty and the continuous spectrum equals $J(f) \backslash \{1\}$.

\vspace{1em}

{\bf Theorem 2}.
In $\ell^1 $, the point spectrum  of $S_p$ is empty.
The residual spectrum  of $S_p$ is not empty and  contains a dense and countable subset of the
Julia set $\partial (J_f)$, i.e. $\bigcup_{n=0}^{+\infty} f^{-n}\{1\} \subset \sigma_{r} (S_p)$. The continuous spectrum is equal to the relative  complement of the
residual spectrum with respect to the filled Julia set $J_f$.

\vspace{1em}

{\bf Theorem 3}.
 The spectra of $S_p$ acting respectively in $\ell^{\infty},\; c_{0},\; c$ and $\ell^{\alpha},\; \alpha \geq 1$, associated to the stochastic Fibonacci adding machines contain the set $\mathcal{E}_{p}= \{ \lambda \in \mathbb{C} \; \vert \; ( \lambda_1, \lambda) \in
J (g)\}$ where $J (g)$ is the filled Julia set of the function $g$ and $\lambda_1=1-p
+\frac{ (1- \lambda- p)^2}{p}.$


 {\bf Conjecture}:  We conjecture that in the case of $\ell^1 $, the residual spectrum of the transition operator associated to the stochastic adding machine in base $2$ is  $\sigma_{r} (S_p)= \bigcup_{n=0}^{+\infty} f^{-n}\{1\} $.
 For Fibonacci stochastic adding machine, we conjecture that the spectra of $S_p$ in the Banach spaces cited in Theorem 3 are  equals to the set $\mathcal{E}_{p}$.
\begin{rem}
 The methods used for the proof of our results can be adapted for a large class of stochastic adding machine given by transductors.
\end{rem}
\begin{rem} 
We point out that from Killeen and Taylor method one may deduce in the case of $\ell^{\infty}$
 that the residual and continuous spectrum is empty. 
On the contrary here we compute directly the residual and continuous spectrum in $\ell^{\alpha}, c_0$ and $c$.
\end{rem}
\begin{center}
\hspace{-9.5 mm}\includegraphics[scale=0.7]{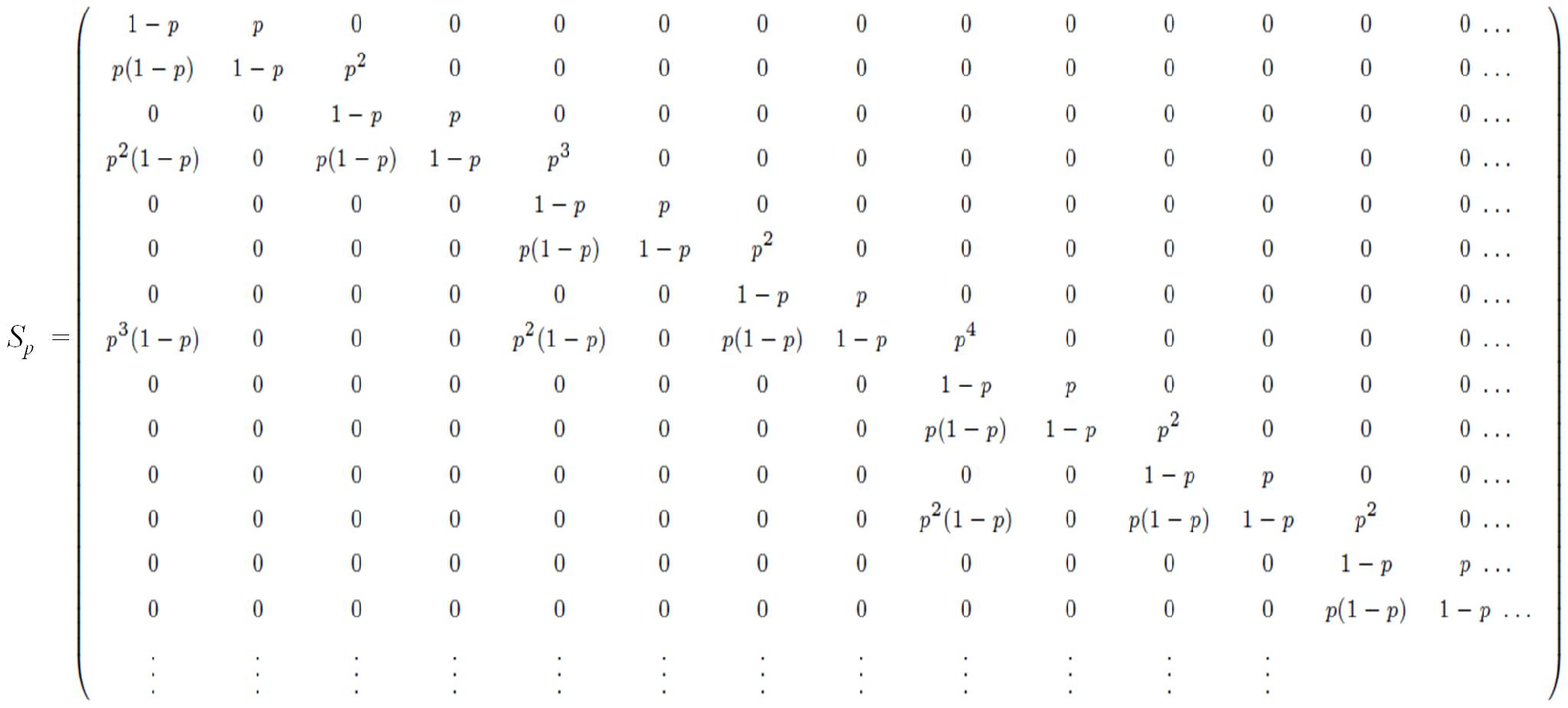}
{\footnotesize {Fig.2. Transition operator of stochastic adding machine in base 2}}
\end{center}
\section{\bf Proof of our main results for stochastic adding machine in base 2.}
We are interested in the spectrum of $S_p$ on three Banach spaces connected by
duality. The space $c_0$  is  the space of complex sequences which
converge to zero,  in other words, the continuous functions on $\N$ vanishing at
infinity. The dual space of $c_0$ is by Riesz Theorem the space of
bounded Borel measures on $\N$ with total variation norm. This space can be
identified with $\ell^1$, the space of summable row vectors. Finally, the
dual space of $\ell^1$ is  $\ell^{\infty}$ the space of
bounded complex sequences.

We are also interested in the spectrum of $S_p$
 as operator on the
space $\ell^{\alpha}$ with $\alpha  > 1$ and also in the space $c$ of complex convergent sequences.


\begin{prop}\label{defia}
The operator $S_p$ (acting on the right) is well defined on the space $X$ where $X \in \{c_0,\; c,\; l^{\alpha},\;
\alpha  \geq 1\}$, moreover $||S_p|| \leq 1$.
\end{prop}

Since the operator $S_p$ is bi-stochastic, the proof of this proposition is a straightforward consequence 
of the following more general lemma. 

\begin{lemm}\label{BanachSteinhaussEasy} Let $A=(a_{i,j})_{i,j \in \N}$ be an infinite matrix with nonnegative
coefficients. Assume that there exists a positive constant $M$ such that
\begin{enumerate}
 \item $\ds \sup_{i \in \N}\left(\sum_{j=0}^{\infty}a_{i,j}\right) \leq M,$
\item $\ds \sup_{j \in \N}\left(\sum_{i=0}^{\infty}a_{i,j}\right) \leq M.$
\end{enumerate}
Then $A$ defines a bounded operator on the spaces $c_0,\; c,\; l^{\infty}$ and  $\ell^{\alpha}(\N)$
 with $\alpha  \geq 1$. In addition the norm of $A$ is less than $M$.
\end{lemm}

\begin{proof}

\noindent By the assumption $(1)$ it is easy to get that $A$ is well defined on $\ell^{\infty}(\N)$ and its
norms is less than $M$.

\noindent Now, let $v= (v_n)_{n \geq 0},\; v \ne 0$ such that $\ds \lim_{n \longrightarrow +\infty}v_n =l \in \mathbb{C}$, then  for any $\varepsilon>0$ there exists a positive integer $j_0$ such that for any
$j \geq j_0$, we have $\displaystyle |v_j-l|  \leq \frac{\varepsilon} {2 M}$. Let $\ds d= \sum_{j=0}^{+\infty}a_{n,j}$, then from the assumption $(1)$, we have that for any $n \in \N$,
\begin{eqnarray}\label{epsilonsur2}
 \left| (Av)_n - d\cdot l\right|=\left|\sum_{j=0}^{+\infty}a_{n,j} (v_j -l)\right|
\leq \displaystyle \sum_{j=0}^{j_0-1}a_{n,j}|v_j -l|+ \frac{\varepsilon}2 .
\end{eqnarray}
\noindent But by the assumption $(2)$, for any $j \in \{0,\ldots,j_0-1\}$, we have
$\ds  \sum_{n=0}^{+\infty}a_{n,j} < \infty$. Then there exists
$n_0 \in \N$ such that for any $n \geq n_0$ and for any $j \in \{0,\cdots,j_0-1\}$,  we have
\begin{eqnarray}\label{2epsilonsur2}
 |{a}_{n,j}|  \leq \frac{\varepsilon}{2 j_0 (\delta+1)} \mbox { where } \delta= sup \{ \vert v_j -l \vert, \; j \in \N\}
\end{eqnarray}
Combined (\ref{epsilonsur2}) with (\ref{2epsilonsur2}) we get that
$$ \left| (Av)_n - d\cdot l \right| \leq \varepsilon,\; \forall n \geq n_0.$$
Hence
$AX \subset X$ if $X= c_0$ or $c$.

\vspace {1em}

\noindent Now take  $\alpha > 1$ and $v \in
l^{\alpha}$. For any integer integer $i \in \N$, we have
\begin{eqnarray*}
\left|(Av)_i\right|^\alpha \leq {\left(\sum_{j=0}^{+\infty}{a}_{i,j}|v_j|\right)}^\alpha.
\end{eqnarray*}
\noindent Let  $\alpha'$ be a conjugate of $\alpha$, i.e, $\ds \frac1{\alpha}+\frac1{\alpha'}=1.$ Then,
by H\"older inequality we get
\begin{eqnarray*}
{\left(\sum_{j=0}^{+\infty}{a}_{i,j}|v_j|\right)}^\alpha \leq {\left(\sum_{j=0}^{+\infty}{{a}_{i,j}}\right)}^{\frac{\alpha}{\alpha'}}
{\left(\sum_{j=0}^{+\infty}{{a}_{i,j}}|v_j|^{\alpha}\right)}.
\end{eqnarray*}
Hence

\begin{eqnarray}
\label{tt}
{\left(\sum_{j=0}^{+\infty}{a}_{i,j}|v_j|\right)}^\alpha \leq {\left(\sup_{l\in \N}\sum_{j=0}^{+\infty}{{a}_{l,j}}\right)}^{\frac{\alpha}{\alpha'}}{\left(\sum_{j=0}^{+\infty}{{a}_{i,j}}|v_j|^{\alpha}\right)}.
\end{eqnarray}

\noindent Thus
\begin{eqnarray*}
||Av||^{\alpha}_{\alpha}
& \leq&  M^{\frac{\alpha}{\alpha'}}
\sum_{i=0}^{\infty} \left( \sum_{j=0}^{+\infty}{{a}_{i,j}}|v_j|^{\alpha} \right)\\
&=& M^{\frac{\alpha}{\alpha'}} \sum_{j=0}^{\infty}\left( \sum_{i=0}^{\infty}
{a}_{i,j} \right)\vert v_j \vert ^{\alpha}\\
&\leq& M^{\frac{\alpha}{\alpha'}} \sup_{j \in \N}\left(
\sum_{i=0}^{\infty}{a}_{i,j}\right) ||v||^{\alpha}_{\alpha}\\
&\leq&  M^{1+\frac{\alpha}{\alpha'}}||v||^{\alpha}_{\alpha}.
\end{eqnarray*}

\noindent{}Then
\begin{eqnarray*}
 \vert \vert A v \vert \vert_{\alpha} \leq   M \vert \vert v \vert \vert_{\alpha}.
\end{eqnarray*}
Hence $A$ is a continuous operator and $\vert \vert A  \vert \vert \leq M$.

The case $\alpha=1$ is an easy exercise and it is left to the reader.

\end{proof}


From Proposition \ref{defia}, we deduce that
 $S_p$ is a Markov operator and its spectrum is contained in the unit disc of complex numbers.

Consider the map $ f: z \in \mathbb{C} \longmapsto
\left(\frac{z-(1-p)}{p}\right)^2$ and denote by $J(f)$ the associated filled Julia set defined by:\\
$$
J(f)=\left \{ z \in \mathbb{C},\;   \vert f^{(n)}(z) \vert
\not \longrightarrow \infty\right\}.
$$
\noindent Killeen and Taylor investigated the spectrum of $S_p$
acting on $\ell^{\infty}$. They proved that the point  spectrum of $S_p$ is equal to the filled Julia set of $f$.
In addition, they showed that the spectrum is invariant under the action of $f$. As a consequence,
one may deduce that the continuous and residual spectra in this case are empty.


Here we will  compute exactly the residual part and the continuous part of
the spectrum of $S_p$ acting on the  spaces $c_0,\; c$  and
$\ell^{\alpha}, \; \alpha  \geq 1$.

\begin{thm}\label{spp}
The spectrum of the operator $S_p$ acting on $X$ where $X \in \{c_0,\; c,\; l^{\alpha},\;
\alpha  \geq 1\}$
is equal to the filled Julia set of $f, \; J (f)$.
 Precisely,
 in $c_0$ (resp. $\ell^{\alpha}, \; \alpha  > 1$), the continuous  spectrum of  $S_p$
is equal to   $J(f)$ and the point and residual spectra are empty.
In $c$, the point spectrum is the singleton $\{1\}$, 
the residual spectrum is  empty and the continuous spectrum is $J(f) \backslash \{1\}$.
\end{thm}

For the proof of Theorem \ref{spp} we shall need the following proposition.

\begin{prop} 
\label{inclu}
The spectrum of $S_p$ in $X$,  where $X \in \{c_0,\; c,\; l^{\alpha},\; 1 \leq \alpha \leq  +\infty\}$, is contained in the filled Julia set of $f$.
\end{prop}

The main idea of the proof of Proposition \ref{inclu} can be found in the Killen-Taylor proof. The key argument is 
that the $\widetilde{S_p}^{2}$ is similar to the operator $ES_p \oplus OS_p$, where 
$\widetilde{S_p}=\ds \frac {S_p- (1-p)Id}{p}$ and $E, O$ denote the even and odd operators acting on $X$ by
$$E (h_0,h_1,\ldots)= (h_0,0,h_1,0, h_2,\ldots),$$  \noindent{}and
$$O (h_0,h_1,\ldots)= (0, h_0,0,h_1,0, h_2,\ldots),$$ for any $h= (h_0,h_1,\ldots)$ in $X$. Precisely, 
for all $v= (v_{i})_{i \geq 0} \in X $,
we have
$$\widetilde{S_p}^{2}(v)= E S_p (v_0, v_2,\ldots v_{2n},\ldots)+ O S_p (v_1, v_3,\ldots v_{2n+1},\ldots).$$

As a consequence we deduce from the mapping spectral theorem \cite{Schechter} that the 
spectrum of  $S_p$ is invariant under $f$. 



\vspace{0.5em}

Let us start the proof of Theorem \ref{spp} by proving the following  result.

 \begin{prop}
 \label{specte}
The point spectrum of $S_p$ acting on $X$ where $X \in \{ c_0,\; l^{\alpha},\; \alpha \geq 1\}$  is empty, and the point spectrum of $S_p$  on $c$ is equal to $\{1\}$.
\end{prop}

For the proof, we need the following lemma from \cite{KT} .

\begin{lemm}\label{Ktlemma}\cite{KT}.
Let $n$ be a nonnegative integer and $X_n= \{m \in \mathbb{N}:\; (S_p)_{n,m} \ne 0\}$, then the following properties are valid.
 \begin{enumerate}
\item For all nonnegative integers  $n$, we have $n \in X_{n}$ and   $(S_{p})_{n,n}= 1-p$.
\item If $n=\varepsilon _{k}\ldots
\varepsilon _{1}0,\; k \geq 2,$ is an even integer then $ X_{n}= \{n, n+1\}$ and  $(S_{p})_{n,n+1}= p$.
 \item If $n=\varepsilon _{k}\ldots
\varepsilon _{t}0\underbrace{1 \ldots 1}_{s}$ is an odd integer with $s \geq 1$
and $k \geq t \geq s+1 $, then  $X_{n}= \{n, n+1, n- 2^m+1,\; 1 \leq m \leq s\}$ and $n$ transitions to  $n+1=
\varepsilon _{k}\ldots \varepsilon _{t}1\underbrace{0 \ldots
00}_{s} $ with probability $(S_p)_{n,n+1}=p^{s+1}$,
 and $n$ transitions to  $n-2^m+1= \varepsilon
_{k}\ldots
\varepsilon _{t}0\underbrace{1 \ldots 1}_{s-m}\underbrace{0
\ldots 0}_{m}$,
$ 1 \leq m \leq s$
 with  probability
$(S_{p})_{n,n-2^m+ 1}= p^{m}(1-p)$.
\end{enumerate}
\end{lemm}

{\bf Proof of Proposition \ref{specte}.}
Let $\lambda$ be an eigenvalue of $S_p$ associated to the eigenvector
$v= (v_n)_{n \geq 0}$ in $X$ where where $X \in \{ c_0, c, \; l^{\alpha},\; \alpha \geq 1\}$
Let $\lambda$ be an eigenvalue of $S_{p}$ associated to the eigenvector
$v= (v_i)_{i \geq 0}$ in $X$. By  Lemma \ref{Ktlemma}, we see that the operator $S_p$ satisfies $(S_{p})_{i,i+k}=0$ for all $i,k
\in \mathbb{N}$ with $k \geq 2$. Therefore, for all integers $k \geq 1,$
we have
\begin{eqnarray}
\label{rrr} \sum_{i=0}^{k}(S_{p})_{k-1,i} v_i = \lambda v_{k-1}.
\end{eqnarray}
 Then, one can prove by induction on $k$ that for all
integers $k \geq 1$, there exists a complex number $q_{k}=q_{k} (p,
\lambda)$ such that
\begin{eqnarray} \label{for3}
v_k= q_k v_0
\end{eqnarray}
By Lemma
\ref{Ktlemma} and the fact that $(S_p- \lambda I)v)_{2^n}=0$ for all  nonnegative integers $n$, we get
\begin{eqnarray}
\label{for2}~~~~~ p^{n+1} v_{2^n}+ (1-p- \lambda)
v_{2^n-1} + \sum_{i=1}^{n}p^{i}(1-p) v_{2^n- 2^i}=0 ,\; \forall n \geq 0.
\end{eqnarray}

\noindent{}Hence $$v_{2^{n}}= \frac{1}{p} A -(\frac{1}{p}-1)v_0,$$
where $A=\displaystyle -\frac{1}{p^n} ( (1-p- \lambda)
v_{2^{n-1}+ (2^{n-1}-1)} + \sum_{i=1}^{n-1}p^{i}(1-p) v_{2^{n-1}+ (2^{n-1}- 2^i)}.$

On the other hand, by the self similarity structure  of the transition matrix $S_{p}$, one can prove that if $i$ and $j$ are two integers such that for some positive integer $n$ we have $2^{n-1} \leq i,j < 2^n$, then
the transition probability from $i$ to $j$ is equal to the transition probability from $i-2^{n-1}$ to $j -2^{n-1}$.
Using this last fact and (\ref{for2}), it follows that
$$ v_{2^n}=  \frac{1}{p} q_{2^{n-1}}v_{2^{n-1}}-\left(\frac{1}{p}-1\right)v_0.$$
\noindent This gives
\begin{eqnarray}
\label{for5} q_{2^n}=  \frac{1}{p} q_{2^{n-1}}^2-\left(\frac{1}{p}-1\right),
\end{eqnarray}
\noindent where
$$ q_{2^0}=q_{1}= -\frac{1-p- \lambda}{p}.$$

{\bf Case 1:  $ v \in c_0$ or $\ell^{\alpha},\; \alpha \geq 1$.}

We have $\lim_{n \to \infty} q_{2^n}= 0$. Thus by (\ref{for5}), we get  $p=1$, which is absurd, 
then the point spectrum is empty.

\vspace{1em}

{\bf Case 2:  $ v \in c$.}
Assume that $\lim q_n= l \in \mathbb{C}$, then by (\ref{for5}), we deduce that $l=1$ or $l= p-1$.
On the other hand, for any $n \in \mathbb{N}$, there exist $k$ nonnegative integers $n_1 < n_2 \ldots <n_k$  such that $n= 2^{n_1}+ \cdots 2^{n_k}$.
We can prove (see \cite{KT}) that
\begin{eqnarray}\label{prod_q}
\label{produit}
q_n= q_{2^{n_1}}\ldots q_{2^{n_k}}.
\end{eqnarray}
Then $\lim q_{ 2^{n-2}+ 2^{n}}= l^2= l$, thus $l=p-1$ is excluded. Since $S_p$ is stochastic, we  conclude that $l=1$ and $\sigma_{pt, c}(S_p)=\{1\}$.

\hfill $\Box$

\begin{rem} By the same arguments as above, Killeen and Taylor in \cite {KT} proved that
the point spectrum of $S_p$ acting on
$\ell^{\infty}$ is equal to the filled Julia set of the quadratic map $f$. In fact, 
it is easy to see from the arguments above that $\sigma_{pt, l^{\infty}} (S_{p})= \{\lambda \in \mathbb{C},\; q_{n}(\lambda) \mbox{ bounded } \}$.
Indeed, (\ref{for5}) implies that if $(q_{2^{n}})_{n \geq 0}$ is bounded, then  for all $n \geq 0,\; \vert q_{2^{n}} \vert \leq 1.$
This clearly forces $\sigma_{pt, l^{\infty}} (S_p)= \{\lambda \in \mathbb{C},\; q_{2^{n}}(\lambda) \mbox { bounded } \}$ 
by \eqref{prod_q}. Now,
since
$$q_{2^{n}}= h \circ f^{n-1} \circ h^{-1} (q_{1})= h \circ f^{n-1}(\lambda), \forall n \in \N,$$
where
$ h(x)= \displaystyle \frac{x}{p}- \frac{1-p}{p}$, we conclude that $\sigma_{pt, l^{\infty}}(S_p)= J(f)$. 
It follows from Proposition \ref{inclu} that  $\sigma _{ l^{\infty}}(S_p)= J(f)$ and the residual and continuous spectra are
empty.

\end{rem}

\begin{prop}\label{Spr} The residual spectrum of $S_p$ acting on $X \in \{ c_0$, $c$,  $\ell^{\alpha},\; \alpha > 1\}$
  is empty.
\end{prop}

\begin{proof}

Let $\lambda $ be an element of the residual spectrum of $S_p$ acting on $c_0$ (resp. $c$). Then,  by Proposition \ref{residual}, we deduce that   there exists a sequence
 $u= (u_k)_{k \geq 0} \in l^1(\N)$  such that $u (S_p- \lambda Id)=0.$

{\bf Claim}. $\displaystyle u_k = \frac{1}{q_k} u_0,\; \forall k \in \mathbb{N}.$\\
  \noindent{}We have
\begin{eqnarray*}\label{les-pairs}
\forall k \in 2\N ,\;
(u (S_p- \lambda Id))_{k+1}=
p u_{k }+ (1-p- \lambda) u_{k+1}=0.
\end{eqnarray*}

\noindent Hence

\begin{eqnarray}\label{parr}
\forall k \in 2\N ,\;
 u_{k}= q_1 u_{k+1}.
\end{eqnarray}

\noindent If $k$ is odd, then $k=2^n -1+ t$ where $t=0$ or $ t= \sum_{j=2}^{s}2^{n_{j}}$ where $1 \leq n <n_2 <n_3,\ldots <n_s$.
Since $(u (S_p- \lambda Id))_{k+1}=0,$ then we have
\begin{eqnarray}
\label{tttl}
p^{n+1}u_{k}+(1-p-\lambda)u_{k+1}+\sum_{i=1}^{n}p^i(1-p)u_{k+2^i}=0.
\end{eqnarray}

Observe that the relation  (\ref{tttl})  between $u_k$  and $u_{k+2^{n}}$ 
is similar to the relation (\ref{for2}) between $v_{2^n}$ and $v_0$. Hence, by induction on $n$, we obtain
\begin{eqnarray}
\label{sd}
u_k= q_{2^{n}} u_{k +2^{n}}.
\end{eqnarray}

Indeed, if $n=1$ then by (\ref{tttl}) and (\ref{parr}), we get
$$p^2 u_k+ \left (q_1(1-p- \lambda\right)+ p(1-p)) u_{k+2}=0.$$
Therefore
$$u_k = \left(\frac{q_1^2}{p}- \frac{1-p}{p}\right)u_{k+2}= q_2 u_{k+2}.$$
Then  (\ref{sd}) is proved for $n=1$.

Now, assume that  (\ref{sd}) holds for the numbers   $1,2,\cdots,m-1$.

\noindent{}Take $n=m$ and $1 \leq i <m$, then $k+2^i= 1+2+ \cdots + 2^{m-1}+ 2^i+ t=  2^i -1+ t'$ where $t'= 2^m+ t$.
Applying the induction hypothesis, we get
$$ u_{k+2^i}= q_{2^i} u_{k+2^{i+1}}= q_{{2^i}}q_{{2^{i+1}}}\ldots q_{{2^{m-1}}}u_{k+2^m}.$$

On the other hand, since $2^{i}+ \cdots + 2^{m-1}= 2^m -2^ i$, we have
\begin{eqnarray}
\label{gf}
 u_{k+2^i}= q_{2^m- 2^i}u_{k+2^m}.
 \end{eqnarray}

 Considering (\ref{tttl}) with $n=m$  and (\ref{gf}) yields

 $$u_k= -
\frac{1}{p^{m+1}} \left(\left(1-p-\lambda\right)q_{2^m- 1}+\sum_{i=1}^{m}p^i(1-p)q_{2^m- 2^i} u_{k+2^m}\right).$$

 Combined (\ref{for3}) and (\ref{for2}), we obtain (\ref{sd}) for $n=m$.
 Then  (\ref{sd}) holds for all integers $n \geq 1$.

In particular, we have $u_{2^{n-1}-1}= q_{2^{n-1}} u_{2^{n}-1}$, for all integers $n \geq 1$.
Thus
\begin{eqnarray}
\label{nnn}
u_{2^{n}-1}= \frac{1} {q_{2^{0}}q_{2} \ldots q_{2^{n-1}}} u_{0}= \frac{1} {q_{2^{n}-1}} u_{0},\; \forall n \geq 1.
\end{eqnarray}

On the other hand, for all integers $n \geq 1$,  by (\ref {parr}) we have 
$u_{2^n}= q_{2^0} u_{{2^n}+2^0}.$ and
from   (\ref{sd}), we see that
\begin{eqnarray}
\label{abc}
 u_{2^n}= q_{2^0} q_{2^1} u_{{2^2}-1+2^n}=\ldots =  q_{2^0} q_{2^1}\ldots q_{2^{n-1}} u_{2^{n+1}-1}.
 \end{eqnarray}
Consequently from (\ref{nnn}) and (\ref{abc}), we obtain
\begin{eqnarray}
\label{xxx}
u_{2^{n}}= \frac{1} {q_{2^{n}}} u_{0},\; \forall n \geq 1.
\end{eqnarray}

\noindent Now fix  an integer $k \in \mathbb{N}$ and assume that $k= \ds \sum_{i=1}^{s}2^{n_{i}}$
where $0 \leq n_1 <n_2 <\ldots <n_s$.
We will prove by induction on $s$ that the following statement holds
\begin {eqnarray}
\label{scs}
u_k= \frac{1} {q_{2^{n_{1}}}q_{2^{n_{2}}}\ldots q_{2^{n_{s}}}} u_{0}=\frac{1} {q_{k}} u_{0}.
\end{eqnarray}

Indeed, it follows from
 (\ref{xxx}), that   (\ref{scs}) is true for $s=1$.

Now assume that(\ref{scs}) is true for all integers $1 \leq i <s$.

{\bf Case 1.} $k$ is odd.

In this case
$\ds k= \sum_{i=1}^{s}2^{n_{i}}= 2^{n}-1+ l = \sum_{j=0}^{n-1}2^{j}+ l$ where $l=0$ if $n=s+1$ and $l= \ds \sum_{i=n}^{s}2^{n_{i}} $ if $n \leq s$.

If $ n \geq 2$, we use (\ref{sd}) to get
$u_{k-2^{n-1}}= q_{2^{n-1}}u_k$ and by induction hypothesis, we have
$$
u_{k} = \frac{1}{q_{2^{n-1}}q_{k-2^{n-1}}}u_0=\frac{1}{q_{k}}
u_0 .$$

If $n=1$, we consider (\ref{parr}) to write
$u_k=\displaystyle \frac{1}{q_1} u_{k-1}$.
Thus, we deduce, by induction hypothesis, that
$$u_k= \frac{1}{q_1 q_{k-1}} u_{0}= \frac{1}{q_{k}}
u_0.$$

{\bf Case 2.} $k$ is even.

In this case  $n_1>0.$ and
by (\ref{parr}), we deduce that  $$u_k= q_{2^{0}} u_{k+2^0}= q_{2^{0}} u_{k+2^1 -1}.$$
Applying (\ref{sd}), it follows that
\begin{eqnarray*}
u_k= q_{2^{0}}q_{2^{1}} u_{k+2^2-1}=\ldots &=&  q_{2^{0}}q_{2^{1}}\ldots q_{2^{n_1 -1}} u_{k+2^{n_1}- 1 }\\
                                      &=&
q_{2^{0}}q_{2^{1}}\ldots q_{2^{n_1 -1}} u_{(k-2^{n_1})+2^{n_1 +1}- 1 }.
\end{eqnarray*}
Hence
\begin{eqnarray*}
u_k &= & q_{2^{0}}\ldots q_{2^{n_1 -1}} q_{2^{n_1 +1}}u_{(k-2^{n_1})+2^{n_1 +2}- 1 }\\
    & =&
q_{2^{0}}\ldots q_{2^{n_1 -1}} q_{2^{n_1 +1}} \ldots q_{2^{n_2 -1}} u_{(k-2^{n_1}- 2^{n_2})+2^{n_2 +1}- 1 }.
\end{eqnarray*}
Thus
\begin{eqnarray}
\label{rts}
u_k= \displaystyle \frac{\ds \prod_{i=0}^{n_s} q_{2^{i}}}{\ds \prod_{i=1}^{s} q_{2^{n_i}}} u_{2^{n_s +1}- 1}.
\end{eqnarray}

By (\ref{rts}) and (\ref{nnn}) we get  $u_k= \ds \frac{1}{\ds \prod_{i=1}^{s} q_{2^{n_i}}} u_{0}= \frac{1}{q_{k}}
u_0.$

Therefore we have proved that for all nonnegative integers
\begin{eqnarray}
\label{dua}
u_k=  \frac{1}{q_{k}}
u_0.
\end{eqnarray}

\noindent{}We conclude that  $u$ is in $\ell^1(\N)$ if and only if
$\ds \sum_{k=1}^{+\infty}\left|\frac1{q_{k}(\lambda)}\right|<\infty.$

\noindent But this gives that  the
residual spectrum  of $S_p$ acting on $c_0$ or $ c$ satisfy
\begin{eqnarray}
\label{sss}
\sigma_{r,C_0}(S_p) \subset \left\{\lambda \in
\overline{\D(0,1)}: \; \sum_{k=1}^{+\infty}\left|\frac1{q_{k}(\lambda)}\right|<\infty\right\}.
\end{eqnarray}

\noindent We claim that $\ds \sum_{k=1}^{+\infty}\left|\frac1{q_{k}(\lambda)}\right|<\infty$ implies
$ \vert q_{2^n-1}\vert \geq 1,$ for all integers $n \geq 1$.
Indeed, by D'Alembert's Theorem, we have
 \begin{eqnarray}
 \label{dalem}
 \lim sup \frac{\vert q_n \vert}{\vert q_{n+1} \vert} \leq 1.
 \end{eqnarray}

\noindent Now assume that $n$ is even. Then $n= 2^{k_0}+\cdots + 2^{k_m}$ where $1 \leq k_0 <k_1 < \ldots <k_m$
(representation in base $2$). In this  case $n+1= 2^{0}+ 2^{k_0}+\cdots + 2^{k_m}$. Using  (\ref{produit}), we obtain
$\ds \frac{\vert q_n \vert}{\vert q_{n+1} \vert} = \frac{1}{\vert q_{1} \vert}$
and  by (\ref{dalem}), we get
\begin{eqnarray}
 \label{aac}
\vert q_{1} \vert \geq  1 .
 \end{eqnarray}

\

 Since for all integers $n \geq 0$, we have
$q_{2^n}= \ds \frac{1}{p} \ds q_{2^{n-1}}^2-\ds \left(\frac{1}{p}-1\right)$. It follows, from the triangle inequality,
that $\vert q_{2^n}\vert \geq 1$ for all integers $n \geq 1$.
Let $i$ be a positive integer. Since $\ds 2^{i}-1= \sum_{j=0}^{i-1}2^{j}$, we obtain by  (\ref {produit}) that
$q_{2^{i}-1}=q_{2^{i-1}}q_{2^{i-2}}\cdots q_1$. Hence
$$\ds \vert q_{2^i-1} \vert \geq 1,\;\mbox { for any integer } i \geq 1.$$
On the other hand, consider the first coordinate of the vector $\mu (S_p- \lambda Id)=0.$ Then we have
$$(1-p- \lambda)\mu_0+ \sum_{i=1}^{+\infty} p^{i}(1-p) \mu_{2^i-1}=0.$$
Dividing the two members of  the last equality by $p$, we obtain
\begin{eqnarray}
\label{qs1}
q_1= \sum_{i=1}^{+\infty} p^{i-1}(1-p) / q_{2^i-1}.
\end{eqnarray}

We claim that there exists an integer $i_0 \in \N$ such that $\vert q_{2^{i_0}-1} \vert >1$. Indeed, if not the series
$\ds \sum_{i \in \N} \frac{1}{\vert q_{2^i-1} \vert}$ will diverge.
\noindent{}Thus
 $\vert q_1 \vert <\ds  \sum_{i \neq i_0}^{+\infty} p^{i}(1-p) +p^{i_0-1}(1-p) <1$.
Absurd. We conclude that the residual spectrum of $S_p$ acting on $c_0$ (resp. $c$ ) is empty.\\
\noindent{}The same proof yields that the
residual spectrum of $S_p$ acting on $\ell^{\alpha},\; \alpha > 1,$ is empty and the proof
of the proposition is complete.
\end{proof}

\begin{rem}
By (\ref {sss}), it follows that $ \lambda$ belongs to  $ \sigma_{r, X}$ where $X=c_0$ or  $c$  or $\ell^{\alpha},\; \alpha >1$, 
implies $\lim \vert q_n (\lambda) \vert = + \infty$.
But this contradicts Proposition \ref{inclu}, which forces $\sigma_{r,C_0}(S_p)=\sigma_{r,l^{\alpha}}(S_p)= \emptyset$.
\end{rem}

\begin{prop}\label{Spc} The following equalities are satisfied:
$$\sigma_{c, c }(S_p)= J(f)\backslash \{1\},\;  \sigma_{c, c_0 }(S_p)=  \sigma_{c, l^{\alpha} }(S_p)= J(f) \mbox {  for all } \alpha >1.$$
\end{prop}

\begin{proof} Assume that $X \in \{c_0,\; c\}$. Then, by Phillips Theorem, we see that the spectrum of $S_p$  in $X$ is equal to the the spectrum of $S_p$  in $\ell^{\infty}$ 
and from Propositions \ref{specte} and \ref{Spr}, we obtain the result.

\noindent Now, assume $X= l^{\alpha},\; \alpha > 1$.
According to Propositions \ref{inclu},  \ref{specte} and \ref{Spr}, it is enough to prove that $J(f) \subset \sigma (S_p) $.
Consider $\lambda \in J(f)$. We will prove that $\lambda$ belongs to the approximate point spectrum of $S_p$ .
  For all  integers $k \geq 2,$ put $w^{(k)}= (1,q_1 (\lambda), \ldots, q_k (\lambda), 0 \ldots 0, \ldots)^{t} \in l^{\alpha} $
where $(q_k (\lambda))_{k \geq 1}= (q_k)_{k \geq 1}$ is the sequence defined in (\ref{for3}) of the proof of Theorem \ref {spp} and
let $u^{(k)}=\ds \frac{w^{(k)}} {\vert \vert  w^{(k)} \vert \vert_{\alpha}}$, then we have the following claim.

\vspace{0.5em}

{\bf Claim:} $\ds \lim_{n\rightarrow+\infty} \vert \vert (S_p- \lambda Id) u^{(2^n)}\vert \vert_{\alpha}=0.$\\

\noindent  Indeed, we have
$$\forall i \in \{0,\ldots,k-1\},  ~~\left ((S_p- \lambda Id) u^{(k)}\right)_i=0.$$
 Thus
\begin{eqnarray*}
\sum_{i=0}^{+\infty} \left \vert {((S_p- \lambda Id) u^{(k)})}_i \right \vert ^{\alpha}
 = \ds
\frac {\ds \sum_{i=k}^{+\infty}\ds \left
\vert  \sum_{j=0}^{k}( S_p - \lambda Id)_{i,j}  w^{(k)}_{j} \right \vert }{ \vert \vert w^{(k)}\vert \vert_{\alpha}^{\alpha}}^{\alpha}.
\end{eqnarray*}

Putting $a_{i,j}= \vert (S_p -\lambda Id)_{i,j}\vert$ for all $i,j$ and using  (\ref{tt}), we get

\begin{eqnarray*}
\left \vert  \sum_{j=0}^{k}( S_p - \lambda Id)_{i,j}  w^{(k)}_{j} \right \vert  ^{\alpha}   \leq C \sum_{j=0}^{k} \vert (S_p - \lambda Id)_{i,j}\vert  \vert w^{(k)}_{j} \vert^{\alpha}
\end{eqnarray*}

where
$C= \ds \sup_{i \in \N} \left(\sum_{j=0}^{\infty} \vert (S_p- \lambda Id)_{i,j} \vert
 \right)^{\frac {\alpha} {\alpha'}}$ and $\alpha'$
is the conjugate of  $\alpha .$

Observe that $C$ is a finite nonnegative constant because $S_p$ is a stochastic matrix and $\lambda$ belongs to $J (f)$ which is a bounded set.

In this way we have

\begin{eqnarray*}
 \left \vert \left \vert {(S_p- \lambda Id) u^{(k)}} \right \vert \right \vert^{{\alpha}}_{\alpha}
& \leq &
C  \sum_{i=k}^{+\infty}\frac{\left( \sum_{j=0}^{k}  \vert w^{(k)}_j \vert^{\alpha}  \vert (S_p- \lambda Id)_{ij}\vert  \right)}
{\vert \vert  w^{(k)}\vert \vert_{\alpha}^{\alpha}}\\
& = &
 \frac{C}{\vert \vert  w^{(k)}\vert \vert_{\alpha}^{\alpha}}  \sum_{j=0}^{k} \vert w^{(k)}_j \vert^{\alpha}
\sum_{i=k}^{+\infty}    \vert (S_p- \lambda Id)_{ij}\vert.
\end{eqnarray*}

\noindent Now, for $k=2^n$, we will compute the following terms $$ A_{kj}=\ds  \sum_{i=k}^{+\infty}  \vert (S_p- \lambda Id)_{ij} \vert,\; 0 \leq j \leq k.$$

Assume that $0 \leq j <k =2^n.$ Then $\left(S_p- \lambda Id\right)_{ij}=( S_p)_{ij}$ for all $i \geq k$.

{\bf Case 1}:  $j$ is odd. Then by Lemma \ref{Ktlemma}, $(S_p)_{ij} \ne 0$ if and only $i=j-1$ or $i=j$.
Hence $(S_p)_{ij}=0$ for all $i \geq k$. Thus
\begin{eqnarray}
\label{t0}
A_{kj}=0.
\end{eqnarray}

{\bf Case 2}:   $j=0$ . Then by Lemma \ref{Ktlemma}, we have

\begin{eqnarray}
\label{tts}
A_{kj}= \sum_{i=2^n}^{+\infty}   (S_p)_{i0}= \ds \sum_{i=n+1}^{+\infty}p^{i}(1-p)= p^{n+1}.
\end{eqnarray}

{\bf Case 3}:   $j $ is even and $j >0$. Then $j= \varepsilon _{n-1}\ldots
\varepsilon _{s}\underbrace{0 \ldots 0}_{s}= \ds \sum _{i=s}^{n-1} \varepsilon_i 2^i$  with $s \geq 1$
and  $\varepsilon _{s}=1$. But by Lemma \ref{Ktlemma}, $(S_p)_{ij} \ne 0$ if and only if $i=2^m -1+j$ where $0 \leq m  \leq s$. Hence $i <2^n= k.$

Therefore, in this case
\begin{eqnarray}
\label{tr}
A_{kj}=0.
\end{eqnarray}

Now assume $j=k=2^n$. In this case, we have
$A_{kj}=\vert 1-p-\lambda\vert+ \ds \sum_{i=2^n+1}^{+\infty}  (S_p)_{i, 2^n}.$
On the other hand, by Lemma \ref{Ktlemma}, we deduce that
$(S_p)_{i, 2^n} \ne 0$ if and only if $i= 2^n+ 2^m-1$ where $0 \leq m \leq n$
and $(S_p)_{2^n+ 2^m-1, 2^n} = p^m (1-p).$
Therefore
\begin{eqnarray}
\label{tt2}
A_{kj}= \sum_{i=2^n}^{+\infty}  \vert (S_p- \lambda Id)_{i,2^n}\vert= \vert 1-p-\lambda \vert+ \sum_{m=0}^{n}  p^m (1-p).
\end{eqnarray}

By (\ref{t0}),(\ref{tts}),(\ref{tr}) and (\ref{tt2}), we have for $k=2^n$ and $0 \leq j \leq k$,
 \begin{eqnarray*}
\label{t3}
A_{kj} \ne 0 \Longleftrightarrow j=0 \mbox { or } j=k=2^n.
\end{eqnarray*}

Consequently

\begin{eqnarray*}
\left \vert \left \vert {(S_p- \lambda Id) u^{(2^n)}} \right \vert \right \vert^{\alpha}_{\alpha}
 & \leq & C~~.      \frac{ \vert w^{(k)}_{0}  \vert ^{\alpha} A_{k0}+  \vert w^{(k)}_{k}  \vert ^{\alpha} A_{kk}} {\vert \vert w^{(k)}\vert \vert_{\alpha}^{\alpha}}    \\%
 & = &   C~~. \frac { p^{n+1} + \vert q_{2^n} \vert^{\alpha} \left( \vert 1-p- \lambda \vert +\ds \sum_{m=0}^{n}  p^m (1-p)\right) }
      {\vert \vert w^{(2^n)}\vert \vert_{\alpha}^{\alpha}}.
    \end{eqnarray*}
\noindent We claim that $\vert \vert w^{(2^n)} \vert \vert_{\alpha} $ goes to infinity as $n$ goes to infinity. Indeed, if not since
the sequence $\vert \vert w^{(2^n)} \vert \vert_{\alpha} $ is a increasing sequence, it must converge. Put
$w=(q_i)_{i \geq 0}$ with $q_0=1$. It follows that the sequence $(w^{(2^n)})_{n \geq 0}$ converges to $w$ in
$\ell^{\alpha}$ which means that there exists a nonzero vector $w \in l^{\alpha}$ such that
$(S_p-\lambda Id)w=0$. This contradicts Proposition \ref{specte}.
Now, since $\lambda$ belongs to the filled Julia set which is a bounded set and $(q_n)_{n \geq 0}$ is a bounded sequence, it follows
that $\vert \vert {((S_p- \lambda Id) u^{(2^{n})})}\vert \vert_{\alpha} $ converge to 0, and the claim is proved.
We conclude that $\lambda$ belongs to the approximate point spectrum of $S_p$ and the proof of Proposition \ref{Spc} is complete.
\end{proof}

This ends the proof of Theorem \ref{spp}.

\vspace{1em}
 

\section*{\bf Spectrum of $S_p$ acting on the right on $\ell^1 $.}

Here, we will study the spectrum  of $S_p$ acting (on the right) in $\ell^1 $.
We deduce from Proposition \ref {inclu} that the Spectrum of $S_p$  on $\ell^1 $ is contained in the filled Julia set $J(f)$.
On the other hand, using the same proof than Proposition \ref{Spc}, we obtain that $J(f)$ 
is contained in the approximate point spectrum of $S_p$. This yields that the spectrum of $S_p$ acting on $\ell^1$ is
equal to $J(f)$.

\begin{thm}
\label{l1}
In $\ell^1 $, the residual spectrum contains a
dense and countable subset of the Julia set $\partial (J(f))$. The continuous spectrum is not empty and
is equal to the relative  complement of the
residual spectrum with respect to the  the filled Julia set $J(f)$.
\end{thm}







\begin{proof}
 The proof of Proposition \ref{Spr}, shows that the residual spectrum of $S_p$ in $\ell^1 $ is equal to the point spectrum  of $S_p$ (acting on right) in ${l^{1}}'  =l^{\infty} $.
 By (\ref{dua}) and (\ref{qs1}),
we see that
$$\sigma_{r}(S_p) =$$
\begin{eqnarray*}
\label{srr}
\left\{\lambda \in \mathbb{C},\; (q_n(\lambda)) {\rm {~and~ }} (1/ q_n (\lambda))  {\rm {~are~bounded~and~ }}
 q_1= \sum_{i=1}^{+\infty} \frac{p^{i-1}(1-p)}{q_{2^i-1}}
   \right \}
\end{eqnarray*}
$$
            =  J (f) \cap \left \{\lambda \in \mathbb{C},\; ( 1/ q_n (\lambda)) \mbox { is  bounded and }
q_1= \sum_{i=1}^{+\infty} \frac{p^{i-1}(1-p)}{q_{2^i-1}}\right \}.
         $$
On the other hand we have

\begin{eqnarray}
\label{fs}
q_{2^n}^{2}= f(q_{2^{n-1}}^2)=\ldots f^{n}(q_1^{2})= f^{n+1}(\lambda),\; \forall n \geq 0.
\end{eqnarray}

Let $n \in \mathbb{N}$ and $ E_n= \{\lambda \in \mathbb{C},\; q_{2^{n}} (\lambda)=1\}$

{\bf Claim 1}:
 $\ds \bigcup_{n=0}^{+\infty} E_n= \bigcup_{n=0}^{+\infty} f^{-n}\{1\}.$

Indeed, let $\lambda \in \mathbb{C}$ such that $f^{n}(\lambda)=1$ for some nonnegative integer $n \geq 1$.
Then, by (\ref{fs}), we have $q_{2^{n-1}}=1$ or $q_{2^{n-1}}=-1$.
From (\ref{for5}), we see that $q_{2^{n-1}}=-1$ implies  $q_{2^{n}}=1.$
Hence $f^{-{n}}\{1\} \subset E_{n-1} \cup E_{n}.$  Since $1 \in E_n$ for all integers $n \geq0$, we conclude that,
$\bigcup_{n=0}^{+\infty} f^{-n}\{1\} \subset  \bigcup_{n=0}^{+\infty} E_n .$ The other inclusion follows from  (\ref{fs}).

{\bf Claim 2}:
$\ds \bigcup_{n=0}^{+\infty} E_n \subset \sigma_r (S_p)$.

Indeed,
assume that $n \in \mathbb{N}$ and $\lambda \in E_n$. Then
by (\ref{for5}), we get that
\begin{eqnarray}
\label{chch}
 q_{2^{k}}= 1, \; \forall k \geq n.
 \end{eqnarray}
But from  (\ref{chch}) and (\ref{produit}), we have that $(q_k(\lambda))_{k \geq 0} \mbox { and } (1/ q_k (\lambda))_{k \geq 0} \mbox { are bounded  }.$
Moreover, we have
\begin{eqnarray*}
 q_1=  \sum_{i=1}^{+\infty} \frac{p^{i-1}(1-p)}{q_{2^i-1}}
&\Longleftrightarrow&  q_2= \sum_{i=2}^{+\infty} \frac{p^{i-2}(1-p) q_1}{q_{2^i-1}}\\
 &\Longleftrightarrow&  q_{2^{k}}=  \sum_{i=k+1}^{+\infty} p^{i-k-1}(1-p)  \frac{ q_{2^{0}} \ldots q_{2^{k-1}}} {  q_{2^i-1}}, \;~~ \forall k \geq 0 \\
                                                    \label{fd}  &\Longleftrightarrow& q_{2^k}=  \sum_{i=k+1}^{+\infty} \frac{ p^{i-k-1}(1-p)} { q_{2^{k} } q_{2^{k+1}}\ldots q_{2^{i-1}}}, \; \forall k \geq 0.
                                                     \end{eqnarray*}
Thus 
$$q_1=  \sum_{i=1}^{+\infty} \frac{p^{i-1}(1-p)} {q_{2^i-1}}
\Longleftrightarrow 1= \sum_{i=0}^{+\infty}p^{i}(1-p).$$
From this $\lambda \in \sigma_r (S_p)$ and the claim 2 is proved.

But $1$ is a repulsor fixed point of $f$, it follows that $\bigcup_{n=0}^{+\infty} f^{-n}\{1\}$ is a dense subset of
the Julia set $\partial J(f)$. By this fact combined with claims 1 and 2, we conclude that the residual spectrum contains a
dense and countable subset of the Julia set $\partial (J(f))$.

On the other hand, $(p-1) ^2 \in J(f)$ since $f((p-1)^2))= (p-1)^2$, but $(p-1) ^2 \not \in \sigma_r(S_p)$ because for any positive integer $n, \; q_{2^n}((p-1) ^2)= (p-1),$
which implies that $\lim q_n= 0$ and hence  $1 / q_n$ is not bounded. Thus $(p-1)^2 \in \sigma_c(S_p)$.
This finishes the proof of the theorem.

\end{proof}

\begin{conj}
We conjecture that  the residual spectrum in $\ell^1 $  equals the set $\bigcup_{n=0}^{+\infty} f^{-n}\{1\}.$
\end{conj}

\vspace{1em}

\section*{\bf Spectrum of $S_p$ acting on the left. }

Phillips Theorem combined with Proposition \ref{residual}, Theorems \ref{spp} and \ref{l1},
leads to the following result.

\begin{thm}
The spectrum of $S_p$ (acting on the left) in the spaces $c_0, c, l^{\alpha}$ where $1 \leq \alpha \leq +\infty$ equals to the filled Julia set $J(f)$.
Precisely:

In $c_0, l^{\alpha}$ where $1 \leq \alpha < +\infty$, the spectrum of $S_p$ equals to the continuous spectrum of $S_p$.

In $c$, the point spectrum of $S_p$ equals $\{1\}$ and the continuous spectrum equals $J(f) \backslash \{1\}$.

In $\ell^{\infty}$, the point spectrum equals to the residual spectrum of $S_p$ in $\ell^1$.

\end{thm}

\begin{center}
\includegraphics[scale=0.5]{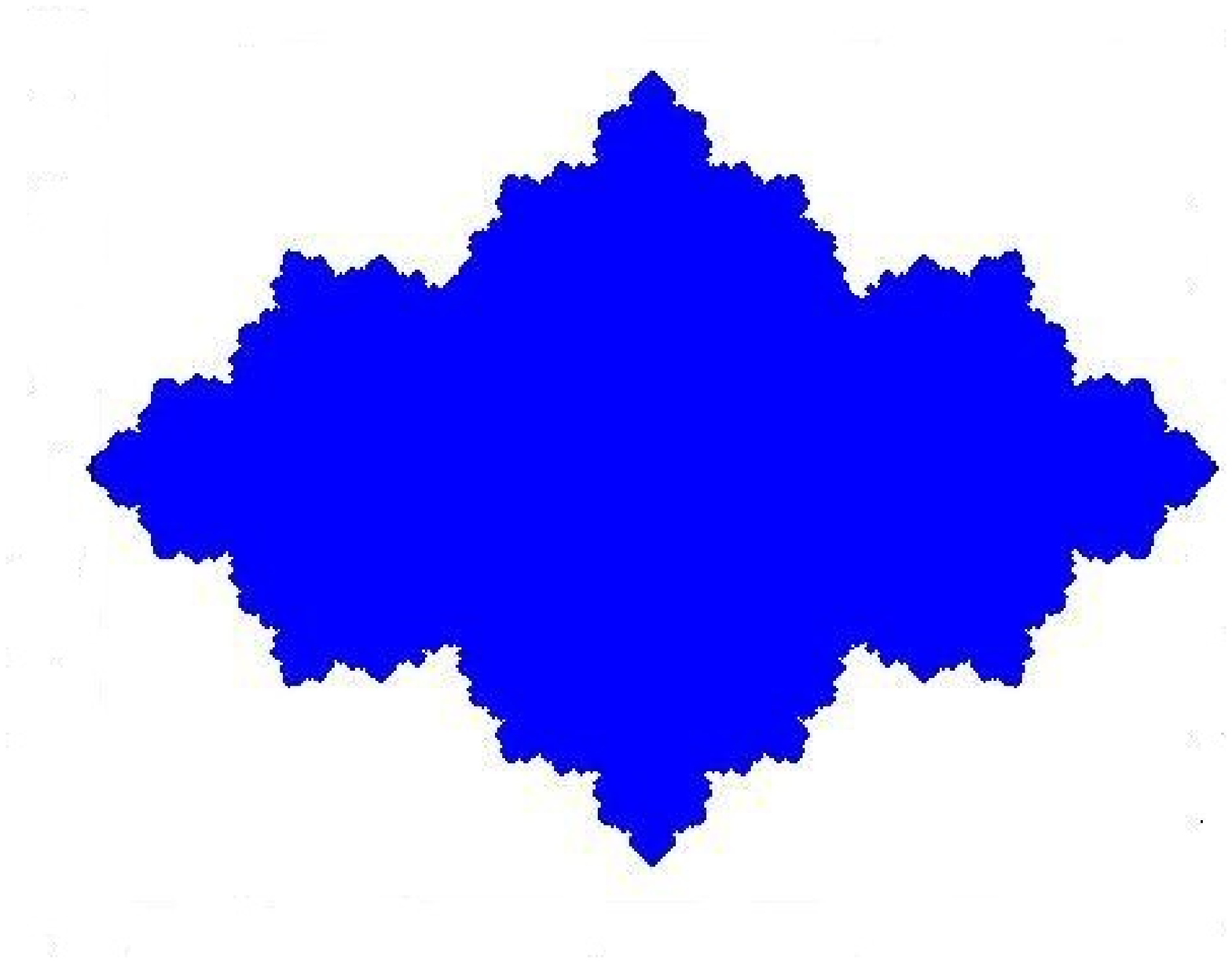}\\
{\footnotesize { \hspace{5.5 em}  Fig.3. Filled Julia set: $p=0.7$}}
\end{center}

\section{ Fibonacci Stochastic adding machine (see \cite{Messaoudi-Smania})}

\vspace{0.5 em}

Let us consider the
Fibonacci sequence $(F_n)_{n \geq 0}$ given by the relation
 $$F_n= F_{n-1}+ F_{n-2} \;\; \forall n \geq 2.$$

Using greedy algorithm, we can write (see \cite{Z})   every
nonnegative integer $N$ in a unique way as $ N=\ds \sum_{i=0}^{k(N)}
\varepsilon_{i}(N)F_i$ where $\varepsilon _{i}(N)= 0$ or $1$ and
$\varepsilon _{i}(N)\varepsilon _{i+1}(N) \ne 0, \; \forall 0 \leq
i \leq k(N)-1$.

It is known that the addition of $1$ in base $(F_n)_{n \geq 0}$
(called Fibonacci adding machine) is given by a finite  state automaton transductor on $A^{*} \times A^{*}$ where $A=\{0,1\}$
(see Fig.4). This transductor is formed by two states ( an
initial state $I$ and a terminal state $T$). The initial state is
connected to itself by $2$ arrows. One of them is labeled by
$(10,00)$ and the other by $(101,000)$. There are also $2$ arrows
going from the initial state to the terminal one. One of these
arrows is labeled by $(00,01)$ and the other by $ (001, 010)$. The
terminal state is connected to itself by $2$    arrows. One of them
is labeled by $(0,0)$ and the other by $(1,1)$.

 Assume that $N=
\varepsilon_{n}\ldots \varepsilon_0$. To find the digits of $N+1$,
we will consider
  the finite path  $c=
(p_{k+1}, a_{k}/ b_{k}, p_{k})  \ldots (p_2, a_1/ b_1, p_1)(p_1, a_0
/ b_0, p_0)$ where $ p_i \in \{I, T\},\; p_0= I,\; p_{k+1}= T,\;
a_i, b_i \in A^{*}$ where $A=\{0,1\}$ and the words $a_k \ldots a_0$ and $b_k \ldots
b_0$ have no two consecutive $1$. Moreover  $\ldots 0\ldots 0 a_k
\ldots a_0= \ldots 0 \ldots 0 \varepsilon_{n}\ldots \varepsilon_0$.

Hence $N+1=\varepsilon'_{n}\ldots \varepsilon'_0$, where
 $$\ldots
0\ldots 0 b_k \ldots b_0= \ldots 0\ldots 0 \varepsilon'_{n}\ldots
\varepsilon'_0.$$

Example: If $N= 10= 10010$ then

$$N \mbox { corresponds
to the path } (T, 1/1, T) \; (T, 00/01, I) \;  (I, 10/00,I)  .$$
Hence $N+1= 10100= 11.$







\begin{center}
\includegraphics[width=15cm,height=6cm,keepaspectratio=true]{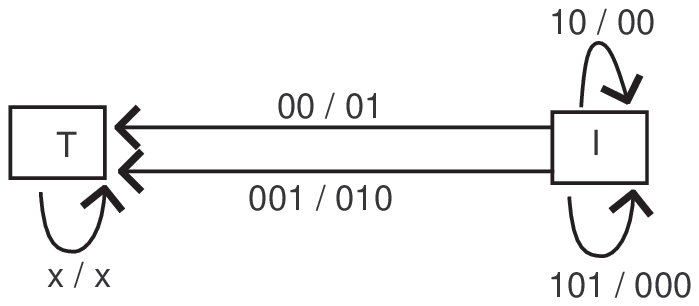}
{ \footnotesize {Fig.4.  Transductor of Fibonacci adding machine
}}
\end{center}

 In \cite{Messaoudi-Smania}, the authors define the stochastic adding machine by the following way:

 Consider "probabilistic"
transductor $\mathcal{T}_{p}$ (see Fig.5) where $0 < p < 1,$ by
the following manner.

The states of $\mathcal{T}_{p}$ are $I$ and $T$. The labels are of the
form $(0/0, 1), (1/1, 1),$\\
$  (a/b, p)$ or $(a/a, 1-p)$ where $a/b$ is a label in $\mathcal{T}$.

 The labeled edges in $\mathcal{T}_{p}$ are of the form $ (T,
(x/x, 1), T)$ where $x \in \{0,1\}$ or of the form $  (r, (a/b, p),
q)$ or $(T, (a/a, 1-p), q)$ where $ (r, a/b , q)$ is a labeled edge
in $\mathcal{T}$, with $q = I$.

The stochastic process $\psi(N)$ is defined  by $\psi(N)=
\sum_{i=0}^{+\infty}r_{i} (N) F_{i}$ where $(r_i(N))_{i \geq 0}$ is an
infinite sequence  of $0$ or $1$ without two $1$ consecutive  and
with finitely many non zero terms.

The sequence $(r_i(N))_{i \geq 0}$ is defined by the following way:

Put $r_i(0)=0$ for all $i$, and assume that we have defined
$(r_i(N-1))_{i \geq 0},\; N \geq 1$. In the transductor $\mathcal{T}_{p}$, consider a path
  $$\ldots (T, (0 /0, 1), T)\ldots (T, (0 /0, 1), T)(p_{n+1}, (a_n / b_n, t_n), p_{n}) \ldots
  (p_1, (a_0 / b_0, t_0), p_0)$$
 \noindent{}where $p_0 = I$ and
 $p_{n+1}= T,$ such that the words  $\ldots r_1(N-1) r_0(N-1)$ and $\ldots 00a_n \ldots
 a_0$ are equal.

 We define the sequence $(r_i(N))_{i \geq 0}$ as the infinite
 sequence whose terms are $0$ or $1$ such that
 $\ldots r_1(N) r_0(N)= \ldots 00b_n \ldots b_0
 .$

  We remark that $\psi(N-1)$ transitions to $\psi(N)$ with  probability of
  $p_{\psi(N-1)\psi(N)}= t_n t_{n-1} \ldots t_0.$

Example 1:  If $N= 10= 10010$, then, in the  transductor of
Fibonacci adding machine, $N$ corresponds to the path $ (T, 1/1, T)
\; (T, 00/01, I) \;  (I, 10/00,I) .$

In the stochastic  Fibonacci adding machine, we have the following
paths (see Figure 2):
 \begin {enumerate}
\item
$(T, (1/1,1), T) \;  (T, (0/0,1), T)(T, (0/0,1), T)\; (T, (10/10,
1-p),I).$ In this case  $N= 10010$ transitions to $10010$ with
probability $1-p$.

\item
 $ (T, (1/1,1), T)\; (T, (00/00,1-p), I)\; (I,
(10/00, p),I)$. In this case  $N= 10$ transitions to $10000= 8$ with
probability $p(1-p)$.

\item
$ (T, (1/1,1), T)\; (T, (00/01,p), I)\; (I, (10/00, p),I)$. In this
case  $N= 10$ transitions to $10100=11$ with probability $p^2$.

\end{enumerate}


\begin{center}
\includegraphics{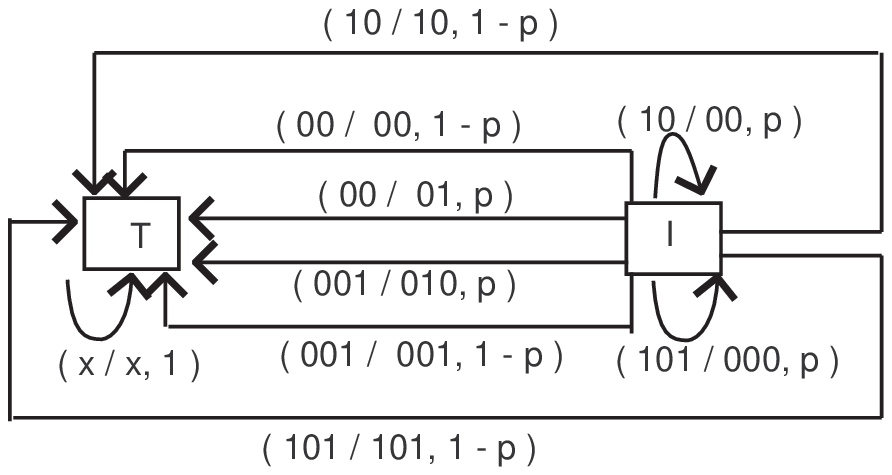}
{\footnotesize {Fig.5. Transductor of Fibonacci fallible adding
 machine }}
\end{center}


\vspace{2em}

By using the transductor $\mathcal{T}_p$, we can  prove the following
result (see \cite {Messaoudi-Smania}).
\begin{prop}
 \label{proba}
Let $N$ be a nonnegative integer, then the following results are
satisfied.
 \begin{enumerate}
\item
$N$ transitions to $N$ with probability $1-p$.
\item
 If $N=\varepsilon _{k}\ldots
\varepsilon _{2}00,\; k \geq 2,$ then $N$ transitions  to $N+1=
\varepsilon _{k}\ldots \varepsilon _{2}01 $ with probability $p$.
 \item
 If $N=\varepsilon _{k}\ldots
\varepsilon _{t}00\underbrace{1010\ldots 1010}_{2s}$ with $s \geq 1$
and $k \geq t \geq 2s+2 $, then $N$ transitions to  $N+1=
\varepsilon _{k}\ldots \varepsilon _{t}01\underbrace{0 \ldots
00}_{2s} $ with probability $p^{s+1}$,
 and $N$ transitions to  $N- \sum_{i=1}^{m}F_{2i-1}= N-F_{2m}+1= \varepsilon _{k}\ldots
\varepsilon _{t}00\underbrace{10 \ldots 10}_{2s-2m}\underbrace{0
\ldots 00}_{2m}$,\\
$ 1 \leq m \leq s$
 with  probability
$p^{m}(1-p).$
\item
If $N=\varepsilon _{k}\ldots
 \varepsilon _{t}0\underbrace{0101\ldots 0101}_{2s},\; s \geq 2$ and $k \geq t \geq 2s+1 $,
then $N$ transitions to  $N+1= \varepsilon _{k}\ldots
 \varepsilon _{t}0\underbrace{1000\ldots 000}_{2s}$ with  probability
 $p^{s}$,
 and $N$ transitions to  $N- \sum_{i=0}^{m}F_{2i}= N-F_{2m+1}+1= \varepsilon _{k}\ldots
\varepsilon _{t}00\underbrace{10 \ldots 10}_{2s-2m}\underbrace{0
\ldots 00}_{2m-1}, \; 2 \leq m \leq s$ with  probability
$p^{m-1}(1-p).$
\item
If $N=\varepsilon _{k}\ldots \varepsilon _{3}001,\; k \geq 3,$ then
$N$ transitions  to $N+1= \varepsilon _{k}\ldots \varepsilon _{3}010
$ with probability $p$.
\end{enumerate}
\end{prop}

By Proposition \ref{proba}, we construct the transition graph. We also find the transition operator $S_p$ associated to the  transition graph.



$$\hspace{-5 mm }\tiny { 
\left(\begin{array}{cccccccccccccccccc}
    1-p& p&0&0&0&0 &0&0 &0&0 &0&0 &0\ldots \\
    0& 1-p & p&0&0 &0&0 &0&0 &0&0&0 &0\ldots \\
    p(1-p) & 0 & 1-p&p^2&0&0&0 &0&0 &0&0 &0 &0 \ldots\\
    0& 0& 0& 1-p&p&0&0 &0&0 &0&0 &0 &0 \ldots\\
    p(1-p)& 0& 0& 0& 1-p&p^2&0&0 &0&0 &0&0 &0  \ldots\\
    0& 0& 0&  0& 0 & 1-p&p&0&0 &0&0 &0 &0 \ldots\\
    0& 0& 0&  0& 0 &0& 1-p&p&0&0 &0&0 &0 \ldots\\
    p^2(1-p)& 0& 0& 0&0& p(1-p)&0& 1-p&p^3&0&0 &0 &0  \ldots\\
    0& 0& 0&  0& 0 & 0&  0& 0 & 1-p&p&0&0 &0 \ldots\\
    0& 0& 0&  0& 0 &0& 0&  0& 0& 1-p&p&0 &0 \ldots\\
    0& 0& 0& 0&0& p(1-p)&0&0& p(1-p)&0& 1-p&p^2 &0  \ldots\\
    \vdots &\vdots &\vdots &\vdots &\vdots &\vdots &\vdots &\vdots
    &\vdots &\vdots &\vdots &\vdots &\vdots
 \end{array}
\right)}
$$$$
{\rm {\footnotesize {Fig.6.~~Transition~graph~of~stochastic~adding~machine~in~Fibonacci~base.}}}
$$

\normalsize

 \vspace{2em}

\begin{rem}
In \cite{Messaoudi-Smania}, the authors prove that
the point spectrum of $S_p $ in
$\ell^{\infty}$ is  equal to the set $ \K_p=  \{ \lambda \in
\mathbb {C}, \;  (q_n (\lambda))_{n \geq 1} \mbox { is bounded }\}$,
 where $q_{F_{0}}(z)= z,\; q_{F_{1}} (z)=z^2,\; q_{F_{k}} (z)= \ds \frac{1}{p}q_{F_{k-1}} (z)q_{F_{k-2}} (z)-\ds\frac{1-p}{p},$ for all $k \geq 2$
and for all nonnegative integers $n$, we have $q_n (z)= q_{F_{k_{1}}} \ldots q_{F_{k_{m}}} $ where $ F_{k_{1}}+ \cdots+ F_{k_{m}}$ is the Fibonacci representation of $n$. 
In particular, $\sigma_{pt(S_{p})}$ is contained in  the set
\begin{eqnarray*}
 \mathcal {E}_p &= &
\{
\lambda \in \mathbb{C} \; \vert \; (q_{F_{n}} (\lambda))_{n \geq 1} \mbox { is bounded } \}\\
                & = &
\{
\lambda \in \mathbb{C} \; \vert \; ( \lambda_1, \lambda) \in
J (g)\}
\end{eqnarray*}
 where  $J(g)$ is the filled Julia set of the function $g:
\mathbb{C}^2  \mapsto \mathbb{C}^2 $  defined by:
$g(x,y)=(\frac{1}{p^{2}}(x- 1+p)(y- 1+p), x)$ and $\lambda_1=1-p
+\frac{ (1- \lambda- p)^2}{p}.$
They also investigated the topological properties of  $\mathcal {E}_p$.

\end{rem}

\begin{prop}
\label{ptfib}
The operator $S_p$ is well defined in the Banach spaces $c_{0},c$
and  $\ell^{\alpha},\; \alpha \geq 1$.
The point spectra of $S_p$ acting in the spaces $c_{0}$,
and  $\ell^{\alpha}$  associated to the stochastic Fibonacci adding machines
  are empty sets. In $c$, the point spectrum equals $\{1\}$.
\end{prop}

\begin{proof}

By Proposition \ref{proba}, we can prove that
  the sum of coefficients of every column of $S_p$ is bounded by a fixed constant $M >0$.

 Indeed, let $n \in \mathbb{N}$ and $s_n= \sum_{i=0}^{+\infty} p_{i,n}$ be the sum of coefficients of the $n$-th column.

 If $n= \varepsilon _{k}\ldots \varepsilon _{2}01 $ or $n= \varepsilon _{k}\ldots \varepsilon _{3}010 $ (Fibonacci representation), then by  1), 2) and 5) of Proposition \ref{proba}, we have $s_n=1$

 If $n= \varepsilon _{k}\ldots \varepsilon _{t}01\underbrace{0 \ldots
00}_{s},\; s \geq 2 $, then  for all integers $i \in \mathbb{N},\;   p_{i,n} >0$ implies that $i= n$ or $i=n-1$ or $i= \varepsilon _{k}\ldots \varepsilon _{t}01  \underbrace {0 \ldots 0}_{s-2m} \underbrace{01 \ldots
01}_{2m},\; s \geq 2m$ or $i= \varepsilon _{k}\ldots \varepsilon _{t}01  \underbrace {0 \ldots 0}_{s-2m} \underbrace{10 \ldots
10}_{2m},\; s \geq 2m$.
Hence $ s_n \leq 1-p + p^ {\lceil 2 \rceil}+  2 \sum_{m=1}^{\infty} p^{m}(1-p) \leq 1+ 2 p.$

If $n=0$, then $s_n \leq 1+p.$

On the other hand, since $S_p$ is a stochastic matrix, then by Proposition \ref{defia}, $S_p$ is well defined in the spaces $c_{0}, \; c$
(resp. in $\ell^{\alpha},\; \alpha \geq 1)$.

Now, let $\lambda$ be an eigenvalue of $S_p$ in $ X$ where $X \in  \{c_{0}, c, \ell^{\alpha},\; \alpha \geq  1\}$  associated to the eigenvector
$v= (v_i)_{i \geq 0} \in X$. Since the
transition probability from any nonnegative integer $i$ to any
integer $i+k,\; k \geq 2$ is $p_{i,i+k}= 0$ (see Proposition
\ref{proba}), the operator $S_p$ satisfies $(S_{p})_{i,i+k}=0$ for all $i,k
\in \mathbb{N}$ with $k \geq 2$. Thus for every integer $k \geq 1,$
we have
\begin{eqnarray}
\label{rrr} \sum_{i=0}^{k}p_{k-1,i} v_i = \lambda v_{k-1}.
\end{eqnarray}
 Then, we can prove by induction on $k$ that for any
integer $k \geq 1$, there exists a complex number $c_{k}=c_k(p,
\lambda)$ such that
\begin{eqnarray} \label{for4}
v_k= c_k v_0
\end{eqnarray}
Using the fact that the matrix $S_p$ is auto-similar, we can prove that $c_k= q_k$ for all integers $k \in \mathbb{N}$ (see Theorem 1, page 303, \cite{Messaoudi-Smania}).
Since $$ q_{F_{n}} (z)= \ds \frac{1}{p}q_{F_{n-1}} (z)q_{F_{n-2}} (z)-\ds\frac{1-p}{p},\; \forall n \in \mathbb{N},$$
and $(q_{F_{n}})$ converges to $0$ when $n$ goes to infinity, we obtain that
the point  spectrum of $S_p$ acting in $c_{0}$
(resp. in $\ell^{\alpha},\; \alpha \geq 1$) is
   empty. Using the same idea than proposition \ref{specte}, we see that $\sigma_{pt, c}= \{1\}$.

\end{proof}

\begin{rem}
By Phillips Theorem and duality, it follows that the  spectra of $S_p$ acting in $X$ where $X \in \{ c_{0},\; c,\; l^{1},\;
 l^{\infty}\}$ associated to the stochastic Fibonacci adding machine are equals.

\end{rem}

\begin{thm}
\label{ptfib}
The  spectra of $S_p$ acting in $X$ where $X \in  \{l^{\infty}, c_{0}, c, l^{\alpha}, \alpha \geq 1\}$   contain the set $\mathcal{E}_{p}= \{\lambda \in \mathbb{C},\;
(q_{F_n}(\lambda) )_{n \geq 0} \mbox { is bounded } \}$.
 \end{thm}

 \begin{proof}

 The proof   is similar to  the proof of Proposition \ref{Spc} and will be done in case $\ell^{\alpha},\; \alpha > 1$.
Let  $\lambda \in \mathcal{E}_p $ and let us prove that $\lambda$ belongs to the approximate point spectrum of $S_p$ in $\ell^{\alpha},\; \alpha > 1$.

\noindent  For every integer $k \geq 2,$ consider $w^{(k)}= (1,q_1 (\lambda), \ldots, q_k (\lambda), 0 \ldots 0, \ldots)^{t} \in l^{\alpha} $
where $(q_k (\lambda))_{k \geq 1}= (q_k)_{k \geq 1}$ is the sequence defined in the proof of Theorem \ref {ptfib}.
Let $u^{(k)}=\ds \frac{w^{(k)}} {\vert \vert  w^{(k)} \vert \vert_{\alpha}}$, then we have the following claim.

\vspace{0.5em}

{\bf Claim:} $\ds \lim_{n\rightarrow+\infty} \vert \vert (S_p- \lambda Id) u^{(F_n)}\vert \vert_{\alpha}=0.$\\
By using the same proof than  Proposition \ref{Spc}, we have

\begin{eqnarray*}
 \left \vert \left \vert {(S_p- \lambda Id) u^{(F_n)}} \right \vert \right \vert^{{\alpha}}_{\alpha}
\leq
\frac{D}{\vert \vert  w^{(F_n)}\vert \vert_{\alpha}^{\alpha}}  \sum_{j=0}^{F_n} \vert w^{(F_n)}_j \vert^{\alpha} B_{F_n,j}
\end{eqnarray*}
where $ D$ is a positive  constant and $ B_{F_n, j}= \sum_{i=F_n}^{+\infty}    \vert (S_p- \lambda Id)_{ij}\vert$.

We can prove by the same manner done in Proposition \ref{Spc} that  for  $0 \leq j \leq F_n$,
 \begin{eqnarray*}
\label{t3}
B_{F_n, j} \ne 0 \Longleftrightarrow j=0 \mbox { or } j=F_n.
\end{eqnarray*}

Indeed,
if $j  \in \{1, \ldots  F_n -1\}$, then since $ i \geq F_n$, we have  $(S_p- \lambda Id)_{ij}= p_{i,j}.$
If the Fibonacci representation of $j$ is  $j= \varepsilon_k \ldots \varepsilon_{2} 01$ or $j= \varepsilon_k \ldots \varepsilon_{t} 1 0 \ldots0$, it is
easy to see by  Proposition \ref{proba}  that $ p_{i,j} \ne 0$ implies $i < F_n$.

On the other hand, if $j=0$ then
$B_{F_n, j}= \sum_{l=F_n}^{+\infty}  p _{l,0} $.
Since $p _{l,0} \ne 0$ if and only $l= F_i -1$ and $p _{F_i -1,0} = p^{\lceil i/2 \rceil} (1-p)$, we have
$B_{F_n, j} \leq 2 \ds \sum_{i=m}^{+\infty}p^{i}(1-p)= 2 p^{m}$ where $m= \lceil (n+1)/2 \rceil.$

Now assume $j=F_n$. In this case, we have
$B_{F_n,j}=\vert 1-p-\lambda\vert+ \ds \sum_{i=F_n+1}^{+\infty}  p_{i, F_n}.$
On the other hand, by Proposition \ref{proba}, we deduce that
$p_{i, F_n} \ne 0$ if and only if $i= F_n+ F_m-1$ where $0 \leq m \leq n$
and $p_{F_n+ F_m-1, F_n} = p^{[m/2]} (1-p).$
Therefore
\begin{eqnarray}
\label{t2}
B_{F_n, F_n}= \vert 1-p-\lambda \vert+  \sum_{m=0}^{n} p^{ \lceil m/2 \rceil} (1-p) \leq \vert 1-p-\lambda \vert+  2
\end{eqnarray}

Hence

\begin{eqnarray*}
\left \vert \left \vert {(S_p- \lambda Id) u^{(F_n)}} \right \vert \right \vert^{\alpha}_{\alpha}
 \leq  D \frac {2 p^{m} + \vert q_{F_n} \vert^{\alpha} \left( \vert 1-p- \lambda \vert +2 \right ) }
      {\vert \vert w^{(F_n)}\vert \vert_{\alpha}^{\alpha}}.
    \end{eqnarray*}
\noindent Since $\vert \vert w^{(F_n)} \vert \vert_{\alpha} $ goes to infinity as $n$ goes to infinity  and $(q_{F_{n}})_{n \geq 0}$ is bounded, it follows
that $\vert \vert {((S_p- \lambda Id) u^{(F_n)})}\vert \vert_{\alpha} $ converge to 0. 
Therefore $\lambda$ belongs to the approximate point spectrum of $S_p$. Thus the spectrum of $S_p$ acting 
on $\ell^{\alpha}\; \alpha >1$,  contains $\mathcal{E}_{p}$.

This finishes the proof for $\ell^{\alpha}\; \alpha >1$. The case $\ell^1$ can be handled in the same way,
the details being left to the reader.

\end{proof}

{\bf Open questions.}
We are not yet able to compute the residual and continuous spectrum of $S_p$ 
acting in the Banach spaces $\ell^{\infty},\; c_{0},\; c$
or  in $\ell^{\alpha},\; \alpha \geq 1$.
We conjecture  that $\sigma(S_{p})= \mathcal{E}_p$. Moreover, in the case of $\ell^{\infty}$  
we conjecture that the residual spectrum  is empty and
 the continuous spectrum is the set $\mathcal{E}_p \setminus \K_{p}.$
 The difficulty here is that the matrix $S_p$ is not bi-stochastic.
 One may also look for a characterization of all real numbers $0 < p <1$ for which $\mathcal{E}_p  \neq \K_{p}.$

\vspace{1em}

\begin{thank}
The first author would like to express a heartful thanks to Albert Fisher, Eduardo Garibaldi,
Paulo Ruffino and Margherita Disertori for stimulating  discussions on the subjet.
It is a pleasure for him to acknowledge the warm hospitality of the
UNESP University in S\~ao Jos\'e of Rio Preto, Campinas University and USP in S\~ao Paulo (Brazil) where a most of this
work has been done.
The second author would like to thanks University of Rouen where a part of this
work has been realized.
\end{thank}





\begin{center}
\includegraphics[width=6cm,height=6cm,keepaspectratio]{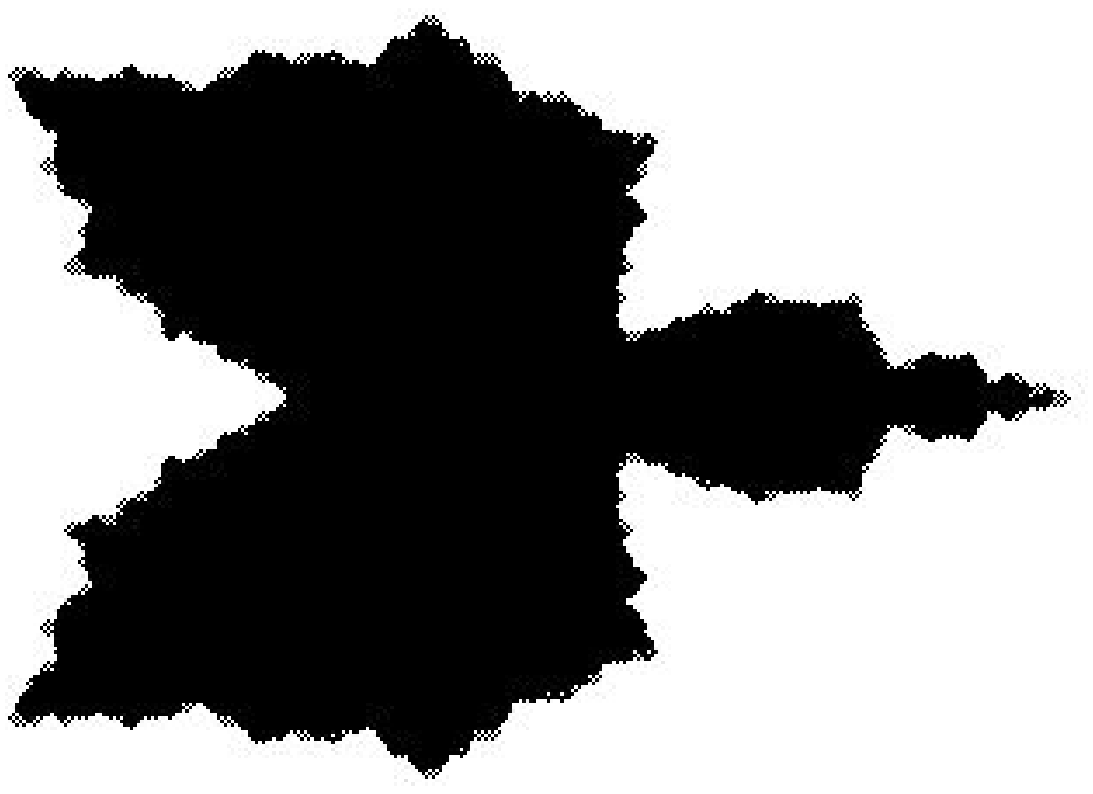}
{\footnotesize \hspace {-13em} { Fig.7. $p=0.625$}}
\end{center}

\begin{center}
\includegraphics[width=6cm,height=6cm,keepaspectratio]{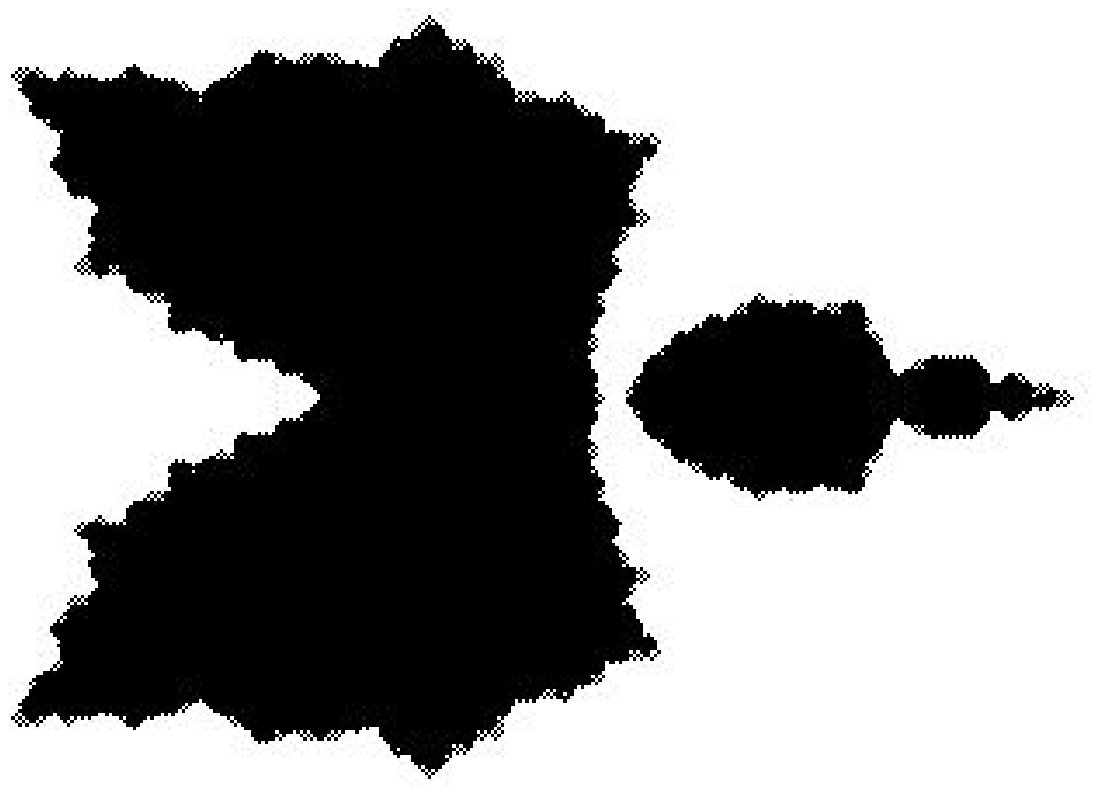}
{\footnotesize \hspace {-13em} { Fig.8. $p=0.621$}}
\end{center}



\begin{thebibliography}{99}
\bibitem{Halmos}
P.~R.~Halmos, {\em A Hilbert space problem book,} Second edition. Graduate Texts
in Mathematics, 19.
Encyclopedia of Mathematics and its Applications, 17. Springer-Verlag, New
York-Berlin, 1982.
\bibitem{KT}
Peter ~R. ~Killeen ~and~ Thomas ~J ~Taylor, {A stochastic adding machine and
complex dynamics},
 {\em{Nonlinearity}}, {\bf 13} (2000),  no. 6, 1889--1903.
\bibitem{KT2}
Peter ~R. ~Killeen ~and~ Thomas ~J ~Taylor,{How the Propagation of Error Through Stochastic Counters Affects Time
Discrimination and Other Psychophysical Judgments,} {\em{Psychological Review}}, {\bf Vol. 107}, No. 3 (2000), 430-459

\bibitem{Krantz}
Steven~~G.~~Krantz, {\em Complex analysis: the geometric viewpoint,} Second
edition. Carus Mathematical Monographs, 23, Mathematical Association of
America, Washington, DC, 2004.

\bibitem{Messaoudi-Smania}
A.~~Messaoudi and D. Smania, {Eigenvalues of Fibonacci stochastic adding
machine}, {\bf Vol 10}, No. 2 (2010), 291-313.
\bibitem{Z}
W. Parry, {On the $\beta$-expansions of real numbers,}  {\em {Acta Math. Acad.
Sci. Hungaray}}, {\bf 11} (1960), 401-416.


\bibitem{Rudin}
W.~Rudin,{\em Functional analysis,} Second edition. International Series in Pure
and Applied Mathematics. McGraw-Hill, Inc., New York, 1991.



\bibitem{Schechter}
M.~Schechter, {\em Principles of functional analysis,} Second edition. Graduate
Studies in Mathematics, {\bf {36}}. American Mathematical Society, 2002.

\bibitem{yoshida}
K.~Yoshida, {\em functional analysis,} Springer Verlag, 1980.

\bibitem {Z}  E.~Zeckendorff, {\em Repr\'esentation des nombres naturels par une somme de
nombres de Fibonacci ou de nombres de Lucas}, Bull. Soc. Royale Sci.
Li\`egege 42 (1972) 179-182.
\end{thebibliography}

\end{document}